\documentclass[a4paper,11pt]{amsart}
\usepackage{amsmath, amssymb, amsthm}
\usepackage{color}
\usepackage[numbers]{natbib}
\usepackage[hyperindex,breaklinks]{hyperref}
\usepackage{xspace}

\providecommand{\U}[1]{\protect\rule{.1in}{.1in}}
\hypersetup{
pdftitle={},
pdfsubject={Probability Theory},
pdfauthor={Jose Blanchet, Karthyek Murthy},
pdfdisplaydoctitle=true,
}
\newtheorem{theorem}{Theorem}

\newtheorem{assumption}{Assumption}

\newtheorem{lemma}{Lemma}

\newtheorem{proposition}{Proposition}
\newtheorem{remark}{Remark}
\theoremstyle{remark}

\hypersetup{colorlinks = true, urlcolor = blue, linkcolor = blue,
  citecolor = blue }

\newcommand{\xmath}[1]{\ensuremath{#1}\xspace}
\renewcommand{\Pr}{\xmath{\mathbb{P}}}
\newcommand{\ve}{{\varepsilon}}
\newcommand{\E}{{\mathbb{E}}}

\begin{document}
\title[Large Delays in G/G/2 Queues]{Tail Asymptotics for Delay in a
  Half-loaded GI/GI/2 Queue with Heavy-tailed Job Sizes}
\author[Blanchet, J.]{Jose Blanchet}
\author[Murthy, K.]  {Karthyek Murthy}
\address{Columbia University, Department of Industrial
  Engineering \& Operations Research, 340 S. W. Mudd Building, 500
  West 120 Street, New York, NY 10027, United States.}
\email{$\{$jose.blanchet, kra2130$\}$@columbia.edu}

\maketitle

\begin{abstract} 
  We obtain asymptotic bounds for the tail distribution of steady-state waiting time in a two server queue where each server processes incoming jobs at a rate equal to the rate of their arrivals (that is, the half-loaded regime). The job sizes are taken to be regularly varying. When the incoming jobs have finite variance, there are basically two types of effects that dominate the tail asymptotics. While the quantitative distinction between these two manifests itself only in the slowly varying components, the two effects arise from qualitatively very different phenomena (arrival of one extremely big job (or) two big jobs). Then there is a phase transition that occurs when the incoming jobs have infinite variance. In that case, only one of these effects dominate the tail asymptotics, the one involving arrival of one extremely big job.  
\end{abstract}

\section{Introduction}
The tail behaviour of the distribution of steady-state delay in
multiserver queues processing jobs with heavy-tailed sizes has
attracted substantial attention in stochastic operations
research. Most of the literature has focused on the case in which the
traffic intensity, $\rho $, (that is, the ratio between the mean
service requirement and the mean interarrival time) is not an integer
and there are qualitative reasons, as we shall discuss, that make the
integer case significantly more delicate to analyze. Our contribution
in this work is to provide the first asymptotic upper and lower bounds
for the tail distribution, that match up to a constant factor, for the
integer case. In that process, we identify the occurrence of a few
surprising phenomena that are not common in the asymptotic analysis of
multiserver queues. We concentrate on the two server queue because it
provides a vehicle to study the qualitative phenomenon that is of
interest to us.

As mentioned earlier, most of the literature concentrates on the case
in which $\rho$ is not an integer. A series of conjectures relating
tail distribution of steady-state delay to the traffic intensity has
been made in \cite{Whitt2000}.  These conjectures turned out to be
basically correct for the case of regularly varying job sizes and were
verified for the case of a two-server queue in \cite{MR2201624}, where
more general asymptotic bounds for subexponential distributions are
provided. In \cite{Foss_MoOR}, the authors provide bounds (up to
constants) that verify the conjecture in \cite{Whitt2000} for general
multiserver queues with regularly varying job sizes and non-integer
traffic intensity. There is a related body of literature aimed at
studying stability properties, such as the existence of the mean
steady-state delay, in terms of the traffic intensity of the system
and tail properties of the incoming traffic. The relations found in
this literature, see \cite{Wolf_Sigman_97} and
\cite{Scheller-Wolf:2011:SST:1900687.1900690}, again are also derived
only for the case of non-integer traffic intensity and are consistent
with the relations found for the tail distributions mentioned earlier
(which can be used to derive the existence of moments).

In order to discuss our contributions in more detail, let us introduce
some notation. Let $V$ denote the amount of time required to service a
generic job arriving to the queue and let $\bar{B}(x) = \Pr\{ V > x
\}.$ We assume that $\bar{B}(\cdot)$ is regularly varying with index
$\alpha > 1,$ that is, 
\begin{equation*}
  \bar{B}(x)=x^{-\alpha }L\left( x\right) ,  
  \label{RV_Ialpha}
\end{equation*}
for some function $L\left( \cdot \right) $ satisfying $%
\lim_{x\rightarrow \infty }L\left( tx\right) /L\left( x\right) =1 $
for each $t>0$; such a function $L\left( \cdot \right) $ is said to be
slowly varying.  Jobs are assumed to arrive as a Poisson stream (or,
more generally, a renewal stream) with rate equal to $\E V$ and
service requirements that are identical copies of $V.$ Under this
setting, the traffic intensity $\rho$ equals 1. Let us write $W$ to
denote the steady-state waiting time of the two-server queue that
processes jobs according to FCFS (first-come-first-serve)
discipline. Our first result establishes that if $\alpha >2$, then
\begin{equation}
  \Pr \left\{W > b\right\}=\Theta \left(
    b^{2}\bar{B}\xspace (b^{2})+b^{2}\bar{B}\xspace^{2}(b)\right) , 
  \label{AS_1}
\end{equation}
as $b \rightarrow \infty.$ Here recall that
$f\left( b\right) =\Theta \left( g\left( b\right) \right)$ if and only
if $f\left( b\right) \leq c_{1}g\left( b\right) $ and
$g\left( b\right) \leq c_{2}f\left( b\right) $ for some positive
constants $c_{1} \text{ and } c_{2}$ that are independent of $b$. To
get a sense of how subtle the difference between the terms appearing
in (\ref{AS_1}) are, it is instructive to consider the example
$L\left( x\right) = \log \left( 1+x\right) $, where the second term
appearing in the right hand side of (\ref{AS_1}) dominates the
asymptotic behaviour. On the other hand, if
$L\left( x\right) = 1/\log \left( 1+x\right) $, the first term in the
right hand side of (\ref{AS_1}) dominates the asymptotic
behaviour. Finally, if $L\left( x\right) \sim c$ for some $c>0$ (the
asymptotically Pareto case) both terms contribute substantially.

Further, let us contrast the result in (\ref{AS_1}) with that derived
in \cite{MR2201624}. For the case $\rho <1,$ it was found that
\begin{equation}
  \Pr \left\{ W > b \right\}=\Theta \left(
    b^{2}\bar{B}\xspace 
    ^{2}(b)\right) ,  \label{RHO_L_1}
\end{equation}
whereas for the case $\rho \in \left( 1,2\right) $, \cite{MR2201624}
obtained that
\begin{equation}
  \Pr \{W > b\}=\Theta \left( b\bar{B}\xspace
    (b)\right),  \label{RHO_G_1}
\end{equation}
as $b \rightarrow \infty$\footnote{From here on, we avoid the
  quantification $b \rightarrow \infty$ whenever it is evident from
  the context}. Since there is a sharp difference between the cases
$\rho < 1$ and $\rho \in (1,2)$ as in \eqref{RHO_L_1} and
\eqref{RHO_G_1}, it has been of great interest to identify what
happens when $\rho$ equals 1. We resolve this in our work by noting
that (\ref{AS_1}) is much closer to the case $\rho <1$ than it is to
the case $\rho >1$. Although, quantitatively, the rates of convergence
between the two terms in (\ref{AS_1}) might differ only by a
multiplicative function which varies slowly, the qualitative picture
behind the mechanism that gives rise to them is dramatically
different. The first term in the right hand side of (\ref{AS_1})
arises from the same type of phenomena behind the tail behaviour in the
case $\rho <1$.

In Section \ref{SEC-RESULT-INTUIT-QUEUES}, in addition to introducing
the notation required to precisely state our results, we discuss at
length the intuition behind both the asymptotic results
(\ref{RHO_L_1}) and (\ref{RHO_G_1}), as well as our asymptotic
expression (\ref{AS_1}). At this point, it suffices to say that the
phenomena underlying the development of (\ref{RHO_L_1}) and
(\ref{RHO_G_1}) are a combination of two features, first, arrival of
large jobs whose effects persist for long time scales, and, second,
the impact of such effects, which is measured using the Law of Large
Numbers. In contrast, the development of (\ref{AS_1}) involves not
only the combination of these two features, but, in addition, one has
to account for the impact of effects which occur at the scales
governed by the Central Limit Theorem.

We identify another interesting phenomenon when the job sizes have
infinite variance: If $\rho =1$ and $\alpha \in \left(1,2\right),$ it
turns out that the asymptotics are governed by
\begin{equation*}
  \Pr \left\{W>b\right\}=\Theta \left( b^{\alpha }\bar{B}
    \xspace(b^{\alpha })\right),
\end{equation*}
suggesting that the tail behaviour is closer to the case $\rho >1$
than to the case $\rho <1.$ This is a sharp transition from the system
behaviour when $\text{Var}[V] < \infty,$ where the tail asymptotic is
closer to the $\rho < 1$ case.  Such surprising transitions in system
behaviour seem to be unique to the integer traffic intensity case. 

In summary, the qualitative development behind our asymptotic bounds
introduces a combination of elements that are not typical in the
asymptotic analysis of multiserver queues. After developing necessary
intuition behind the results \eqref{AS_1}, \eqref{RHO_L_1} and
\eqref{RHO_G_1} in Section \ref{SEC-RESULT-INTUIT-QUEUES}, we derive
the respective lower and upper bounds in \eqref{AS_1} in Sections
\ref{SEC-LB} and \ref{SEC-UB1}. Apart from unraveling surprising
transitions in the system behaviour that seem to happen only when the
traffic intensity is an integer, an important contribution of this
paper is in the use of regenerative ratio representation and Lyapunov
bound techniques to characterize tail behaviour of steady-state delay
in multiserver queues. An alternate proof for the upper bound, that
takes inspiration from a completely different approach due to
\cite{MR2201624} and \cite{Foss_MoOR}, is reported in
\cite{karthyek_thesis}. However, in \cite{MR2201624} and
\cite{Foss_MoOR}, it is crucial to have $\rho $ not equal to an
integer so that certain upper bound processes might be defined.  So,
we believe that our alternate approach presented in
\cite{karthyek_thesis} might add useful ideas to the traditional
techniques used in the asymptotic analysis of multiserver queues.

\section{The main result and its intuition}
\label{SEC-RESULT-INTUIT-QUEUES}
We consider a two-server queue that processes incoming jobs under the
first-come-first-serve discipline. Jobs are indexed by the order of
arrival. Job $0$ arrives at time $0,$ and for $n \geq 1,$ job $n$
arrives at time $T_{1}+\ldots +T_{n}$. Job $n$ requires service for
time $V_n.$ Here the sequence of interarrival times $(T_{n}:n\geq 1)$
and service times $(V_n: n \geq 0)$ are taken to be i.i.d. copies,
respectively, of the generic interarrival and service time variables
$T$ and $V.$ As mentioned in the Introduction, we assume that
$\E T=\E V,$ and hence the traffic intensity $\rho,$ which is the
ratio between $\E V$ and $\E T,$ equals 1. To make the computations
easier, we assume, without loss of generality, that $\E T = 1$
(otherwise, time can always be rescaled to make this hold).
Additionally, we make the following assumptions on the distributions
of $V$ and $T.$
\begin{assumption}
\label{ASSUMP-DIST-V}
The tail distribution of $V$ admits the representation,
\[\bar{B}(x):=\Pr\{V > x\} = x^{-\alpha }L\left( x\right),\] for some
$\alpha >1$ and a function $L\left( \cdot \right) $ slowly varying at
infinity, that is, $\lim_{x \rightarrow \infty} L(tx)/L(x) = 1$ for
every $t > 0.$
\end{assumption}
\begin{assumption}
\label{ASSUMP_LEFT_TAIL}
$\Pr \{T>x\}=o(\bar{B}(x))$.
\end{assumption}
\noindent Assumption \ref{ASSUMP_LEFT_TAIL} is quite natural given
that typically one models interarrival times as exponentially
distributed random variables. We also use the notation
\[X_{n+1}=V_{n}-T_{n+1} \text{ for } n\geq 0.\]
Since $T$ is non-negative, the right-tail of $X := V-T$ is
asymptotically similar to that of $V$ (see, for example, Corollary
1.11 in Chapter IX of \cite{asmussen2000ruin}). In other words,
\begin{align}
  \label{tail_equiv_F_B}
  \Pr \left\{ X > x \right\} \sim \bar{B}(x) \text{ as } x \rightarrow
  \infty. 
\end{align}

The ordered workload vector of the servers as seen by the $n^{th}$ job
during its arrival, denoted by
${\bf W}_{n}=(W_{n}^{(1)}\xspace,W_{n}^{(2)}\xspace ),$ satisfies the
well-known Kiefer-Wolfowitz recursion:
\begin{subequations}
\begin{align}
  W_{n+1}^{(1)}& =\left( W_{n}^{(1)}\xspace+V_{n}-T_{n+1}\right)^{+}\wedge
  \left( W_{n}^{(2)}\xspace-T_{n+1}\right)^{+} \text{ and }  \label{MIN-REC}\\
  W_{n+1}^{(2)}& =\left(
    W_{n}^{(1)}\xspace+V_{n}-T_{n+1}\right)^{+}\vee \left(
    W_{n}^{(2)}\xspace-T_{n+1}\right)^{+}.  \label{MAX-REC}
\end{align}
\end{subequations}
Since $\rho <2,$ the queue is stable in the sense that the weak limit
(limit in distribution) of ${\bf W}_{n},$ denoted by
${\bf W}_{\infty },$ exists and we are interested in deriving bounds
for the tail probabilities of the steady-state waiting time
\begin{equation*}
  \Pr \left\{W_{\infty }^{(1)}\xspace>b\right\}=\lim_{n\rightarrow \infty }\Pr \left\{
    W_{n}^{(1)}\xspace>b\right\} ,
\end{equation*}
for large values of $b.$ Our main result is the following.
\begin{theorem}
  Suppose that $\rho =1 $ and Assumptions \ref{ASSUMP-DIST-V} and
  \ref{ASSUMP_LEFT_TAIL} are in force. If $\alpha >2$, then
\begin{equation}
  \Pr \left\{ W_{\infty }^{(1)}>b\right\} =\Theta \left( b^{2}\bar{B}\xspace
    (b^{2})+b^{2}\bar{B}\xspace^{2}(b)\right) ,\text{ as }b\rightarrow \infty.
\label{FV}
\end{equation}
If $\alpha \in (1,2),$  under the additional assumption that
$\bar{B}(x) \sim cx^{-\alpha}$ for some $c > 0,$  we have that
\begin{equation}
  \Pr \left\{ W_{\infty }^{(1)}>b\right\} =\Theta
  \left( b^{\alpha }\bar{B} \xspace(b^{\alpha })\right) ,\text{ as
  }b\rightarrow \infty .  \label{IV}
\end{equation}
\label{THM-ASYMP}
\end{theorem}
\noindent We now proceed to discuss how this result contrasts with
what is known in the literature and thereby expose the intuition
behind it.

\subsection{Discussion of earlier results in the literature} 
As indicated in the Introduction, the tail asymptotics of steady-state
delay is known depending on the case $\rho <1 $ (or)
$\rho \in \left( 1,2\right),$ and is given by (\ref{RHO_L_1}) and
(\ref{RHO_G_1}), respectively. In order to see the mechanism behind
these two asymptotics, let us assume without loss of generality that
$\E T = 1$ (if not, time can be rescaled to make this assumption
hold). Additionally, let us assume that the generic interarrival time
$T$ has unbounded support (for example, $T$ is exponentially
distributed), and consider the regenerative ratio representation
\begin{equation}
  \mathbb{P}\left\{ W_{\infty }^{(1)}>b\right\} =\frac{\mathbb{E}_{\bf
      0}\left[ \sum_{k=0}^{\tau _{0}-1}I\left( W_{k}^{\left( 1\right) }>b\right) \right] }{
    \mathbb{E}_{\bf 0}\left( \tau _{0}\right) },  
\label{REG-RAT-REP}
\end{equation}
where $\tau _{0}=\inf \{n\geq 1:W_{n}^{\left( 2\right) } = 0\}$
denotes the first time when the Kiefer-Wolfowitz process ${\bf W}_n$
enters the set $\{(0,0)\}.$ Since $\rho < 2$ and $T$ has unbounded
support, the state $(0,0)$ is recurrent, thus leading to the
regenerative ratio representation \eqref{REG-RAT-REP} 
For simplicity, throughout our discussions, we shall assume that $T$
has unbounded support. This assumption is merely technical. It can be
relaxed at the price of using a slightly more complicated regenerative
representation. For further details on the representation
\eqref{REG-RAT-REP} and details on relaxing the assumption on support
of $T,$ see, for example, \cite{borovkov1984asymptotic},
\cite{foss1986method}, \cite{kalashnikov1980stability},
\cite{kalashnikov1990mathematical}, or \cite{foss1991regeneration}.
Moreover, our alternate proof of the upper bound presented in
\cite{karthyek_thesis} does not rely on this assumption.

In order to study \eqref{REG-RAT-REP}, define the stopping times
\begin{align*}
  \tau _{b}^{\left( i\right) } &:=\inf \{n\geq 0:W_{n}^{\left( i \right)
                                   }>b\}, \quad i = 1,2.
\end{align*}
First, let us consider the event
$\{\tau _{b}^{\left( 1\right) }<\tau _{0}\}$, which is the event that
there is at least one customer who waits more than $b$ units of time
in a busy period. Moreover, since
$\tau _{b}^{\left( 2\right) }<\tau _{b}^{\left( 1\right) },$ it is
instructive to first consider the event
$\{\tau _{b}^{\left( 2\right) }<\tau _{0}\}$, which can be seen,
intuitively, to be caused by the arrival of a big job of size larger
than $b$ within the initial $O\left( 1\right) $ units of time in the
busy period. Due to this reasoning, one can write
\begin{align}
  \Pr_{\bf 0} \left\{ \tau_b^{(2)} < \tau_0\right\} = \Theta \left( \Pr\{ V
  > b\} \right) \text{ and } \Pr\left\{ W_{\tau_b^{(2)}}^{(2)} > x
  \left \vert \frac{}{} \right. \tau_b^{(2)} < \tau_0 \right\}
  \approx \Pr\left\{ V > x \left \vert \frac{}{} \right. V > b \right\}.
  \label{APPROX-AFTER-JUMP1}
\end{align}
Therefore, one can approximately characterize the process ${\bf W},$
immediately after the arrival of the first big job of size larger than
$b,$ as below:
\begin{align}
  \frac{1}{b}{\bf W}_{\tau_b^{(2)}} = \frac{1}{b}\left(
  W_{\tau_b^{(2)}}^{(1)},  W_{\tau_b^{(2)}}^{(2)} 
  \right) \approx \left( 0, Z \right), 
  \label{SNAPSHOT-AFTER-JUMP1}
\end{align}
where $Z$ satisfies
$\Pr\{ Z > x\} = \lim_{b \rightarrow \infty} \Pr\left\{ V > bx \left
    \vert \frac{}{} \right. V > b \right\} = x^{-\alpha} \text{ for }
x \geq 1.$
As per recursions \eqref{MIN-REC} and \eqref{MAX-REC}, the server that
gets to process this big job cannot process any new arrivals until
both the workloads become comparable again at some time in the future,
which we refer as $\tau_{eq}.$ During this period where one of the
servers is effectively blocked from processing new arrivals (call it
the blocked server and the other server as active server), the
dynamics of the queue is given by:
\begin{align*}
  {\bf W}_n = \left( \left( W_{n-1}^{(1)} + V_{n-1}-T_n \right)^+,
  W_{n-1}^{(2)} - T_n \right), \quad \tau_b^{(2)} < n < \tau_{eq}. 
\end{align*}
The dynamics of the active server matches with that of the single
server queue, and hence the waiting time experienced by the $k^{th}$
job after the big jump can be be roughly approximated, in
distribution, by maximum of $k$ steps of a random walk with increments
that are i.i.d. copies of $V-T.$ Observe that the aforementioned
random walk has drift equal to $\rho-1,$ which can be positive, zero
(or) negative, respectively, based on whether $\rho > 1, \rho =1$ (or)
$\rho < 1.$ As a consequence, the maximum of the random walk, in the
respective cases, can be of magnitude $O(k), O(\sqrt{k})$ (or) $O(1)$
in $k$ units of time (this can be seen by invoking Law of Large
Numbers and Central Limit Theorem for i.i.d. sums). Therefore, due to
\eqref{SNAPSHOT-AFTER-JUMP1}, the workload until time $\tau_{eq}$ can
be approximately written as
\begin{align}
  \label{W_WHEN_BLOCKED}
  W_{\tau_b^{(2)} + k }^{(1)} \approx \begin{cases} 
    c_1 k \quad\quad \text{ if } \rho > 1,\\
    c_2 \sqrt{k} \quad \text{ if } \rho = 1,\\
    O(1) \quad\  \text{ if } \rho < 1
 \end{cases}
\text{ and } W_{\tau_b^{(2)}+k}^{(2)} \approx bZ - k 
\end{align}
for some positive constants $c_1$ and $c_2.$ Because of this clear
difference in behaviour of $W^{(1)}$ based on the value of $\rho,$ we
need to consider cases $\rho \in \left( 1,2\right) ,\rho <1$ and
$\rho = 1$ separately. We once again stress that our discussion in
this section is completely heuristic, aiming to emphasize the
intuition behind the results. While cases $\rho \in (1,2)$ and
$\rho < 1$ are treated rigorously in \cite{MR2201624}, future sections
in this paper are devoted to the rigorous treatment of the case
$\rho = 1.$

\subsubsection{Case 1: $\rho \in (1,2)$}
If $\rho \in \left( 1,2\right),$ then one server is not enough to keep
the system stable. As a result, when one server is blocked for $O(bZ)$
units of time due to the arrival of a big job, the active server
effectively becomes a single server processing all the arrivals, and
hence the workload $W^{(1)}$ gradually increases with time as in
\eqref{W_WHEN_BLOCKED}. Recall that $\tau_{eq}$ is the time where both
the servers have roughly equal workload, and therefore due to
\eqref{W_WHEN_BLOCKED}, we solve for $\tau_{eq}$ by setting
\begin{align*}
  c_1 \left(\tau_{eq} -
  \tau_b^{(2)}\right)  \approx bZ - \left(\tau_{eq} -
  \tau_b^{(2)}\right).
\end{align*}
As a result, $W^{(1)}$ increases roughly up to time
\begin{align*}
  \tau_{eq} \approx \tau_b^{(2)} + \frac{bZ}{c_1 + 1}, 
\end{align*}
when both $W^{(1)}$ and $W^{(2)}$ become comparable, after which both
the servers jointly process incoming arrivals according to
\eqref{MIN-REC} and \eqref{MAX-REC}, resulting in a total decrease of
workload at rate $2-\rho.$ In this mechanism, for any job to be
delayed by more than $b$ units of time, it must happen that
$c_1 k \geq b$ for some $k \leq bZ/(c_1 + 1),$ and therefore,
\begin{align}
  \lim_{b \rightarrow \infty}\Pr_{\bf 0}\left\{ \tau_b^{(1)} < \tau_0 \ \left
  \vert \frac{}{} \right. \ 
  \tau_b^{(2)} < \tau_0\right\} = \lim_{b \rightarrow
  \infty}\Pr\left\{ c_ 1 \frac{bZ}{c_1 + 1} \geq b\right\} = \Pr
  \left\{ Z > 1 + \frac{1}{c_1}\right\} >   0. 
  \label{SEC-JUMP-PROB}
\end{align}
\noindent If we let $N_1$ to denote the number of jobs that experience
at least $b$ units of delay up to time $\tau_{eq}$ and $N_2$ to denote
the respective count after $\tau_{eq},$ then the above heuristics
suggest that
\begin{align*}
  N_1 &= \frac{ \left(W_{\tau_{eq}}^{(1)} -b\right)^+}{c_1} = \left(
        \frac{bZ}{c_1 + 1} - \frac{b}{c_1}\right)^+ \text{ and  }\\
  N_2 &= \frac{\left( W_{\tau_{eq}}^{(1)} - b\right)^+}{2-\rho} =
        \frac{1}{2-\rho} \left( \frac{c_1bZ}{c_1 + 1} - b \right)^+.
\end{align*}
Therefore, due to \eqref{SEC-JUMP-PROB}, we obtain that
\begin{align*}
  \E_{\bf 0} \left[ \sum_{k=0}^{\tau_0-1} I\left( W_k^{(1)} >
  b\right)\right] &=   \E_0 \left[ \sum_{k=0}^{\tau_0-1} I\left(
                    W_k^{(1)} > b\right) \ \left \vert \frac{}{} \right. \ \tau_b^{(1)}
                    < \tau_0  \right] \times \Pr\left\{ \tau_b^{(1)} < \tau_0 \right\}\\
                  &\approx \E \left[ N_1 + N_2 \ \left \vert \frac{}{} \right. Z > 1 +
                    \frac{1}{c_1} \right] \times \Pr \left\{ Z > 1 + \frac{1}{c_1}
                    \right\} \times \Theta \left( \Pr \left\{ V >
                    b\right\} \right)\\
                  &= \Theta \left( b \times
                    \Pr \left\{ Z > 1 + \frac{1}{r}\right\} \times \bar{B}(b) \right).
\end{align*}
As a result, from \eqref{REG-RAT-REP}, we obtain that
\begin{align*}
  \Pr\left\{ W_\infty^{(1)} > b \right\} &= \Theta\left( b\bar{B}(b) \right),
\end{align*}
which is precisely same as \eqref{RHO_G_1}. This final form of
asymptotic is rigorously established in \cite{MR2201624}, albeit,
using a different reasoning.

\subsubsection{Case 2: $\rho < 1$} If $\rho < 1,$ conditional on the
occurrence of $\{ \tau_b^{(2)} < \tau_0\},$ it is no longer true that
the event $\{ \tau_b^{(1)} < \tau_0\}$ happens with positive
probability as $b \rightarrow \infty$ (compare this with
\eqref{SEC-JUMP-PROB} when $\rho \in (1,2)$). The reason is that if
$\rho <1$, the system is stable and the workload remains $O(1),$ as in
\eqref{W_WHEN_BLOCKED}, even if one removes one server and force it to
operate as a single server system.  As a result, we need to invoke
heavy-tailed large deviations behaviour, which dictates that arrival
of one more job of size larger than $b$ is required, typically, to
experience waiting time larger than $b.$ This requirement is dealt as
follows: Conditional on the occurrence of $\{\tau_b^{(2)} < \tau_0\},$
as in \eqref{W_WHEN_BLOCKED}, we have
  \begin{align*}
    W_{\tau_b^{(2)} + k}^{(1)} = O(1) \text{ and }
    W_{\tau_b^{(2)}+k}^{(2)} \approx bZ - k. 
  \end{align*}
  Here, the workload $W^{(2)}$ becomes smaller than $b$ if
  $k > b(Z-1),$ and therefore, the cheapest way to observe large
  delays (of duration at least $b$) is to have a $K \leq b(Z-1)$ such
  that the $(\tau_b^{(2)} + K)^{th}$ job requires service for duration
  larger than $b.$ 
  Following the same line of reasoning behind
  \eqref{SNAPSHOT-AFTER-JUMP1}, we approximate the size of the second
  big job by $b\hat{Z},$ where $\hat{Z}$ is an independent copy of
  $Z.$ As a result, we arrive at the following distributional
  approximation :
  \begin{align*}
    {\bf W}_{\tau_b^{(1)}}^{(1)} \approx \min \left( bZ-K_1,
    b\hat{Z} \right) \text{ and }
    {\bf W}_{\tau_b^{(1)}}^{(2)} \approx \max \left( bZ-K_1,
    b\hat{Z} \right). 
  \end{align*}
  Next, the number of jobs that get delayed by more than $b$ units of
  time (which depends on $K$) is approximately given by
  \[ N(K) := \frac{\min(bZ-K, b\hat{Z})-b}{2-\rho}, \]
  where $K \leq b(Z-1).$ As a result,
  \begin{align*}
    &\E_{\bf 0} \left[ \sum_{k=0}^{\tau_0-1} I\left( W_k^{(1)} >
      b\right)\right] \approx   \E \left[
      N \left( K\right) I\left(0 \leq K \leq
      b(Z-1) \right) \ \left \vert \frac{}{}
      \right. \tau_{b}^{(2)} < \tau_0\right] \times \Pr\left\{ \tau_b^{(2)} <
      \tau_0 \right\}\\ 
    &\quad\quad= \E \left[ \sum_{k=1}^{b(Z-1)}  \min\left(
      b(Z-1)-k, b(\hat{Z} - 1)\right) I\left(
      V_{\tau_b^{(2)} + k} > b \right) \ \left \vert \frac{}{}
      \right. \tau_{b}^{(2)} < \tau_0 \right] \times
      \Theta\left( \Pr\{ V > b\} \right)\\
    &\quad\quad= \Theta \left( b^2 \Pr\left\{ V > b\right\}^2 \right),  
  \end{align*}
  and therefore,
  $\Pr\{ W_{\infty}^{(1)} > b\} = \Theta(b^2 \bar{B}^2(b)),$ which
  coincides with \eqref{RHO_L_1}.

\subsection{Intuitive discussion of Theorem 1: The case ${\bf \rho =
    1}$} 
Our goal in this discussion is to communicate the following insights:
\begin{itemize}
\item[1)] Contrary to Case 1 and Case 2, the conditional distribution
  of the Kiefer-Wolfowitz vector ${\bf W}$ given that
$\{\tau _{b}^{\left( 1\right) }<\tau _{0}\}$\ does not fully explain
the mechanism behind the asymptotic results in Theorem 1.

\item[2)] Unlike Cases 1 and 2, it is not enough to account for the
    impact of the large service times using linear dynamics which
    evolve according to the Law of Large Numbers.
\end{itemize}
\noindent We shall first concentrate on the situation where the job
sizes $V$ have finite variance, more precisely, the case $\alpha
>2$. The case $\alpha \in \left( 1,2\right)$ can be understood using
similar ideas. We shall leverage off the type of arguments that were
given for Case 1 and Case 2. Since $\rho =1$ sits right in the middle
we shall consider two mechanisms, one involving two jumps (analogous
to Case 2), and one involving one jump (analogous to Case 1).

\noindent \textit{Delays due to two jumps:} Conditional on
$\{ \tau_b^{(2)} < \tau_0\},$ similar to cases 1 and 2, the dynamics
of the active server and the blocked server, as in
\eqref{W_WHEN_BLOCKED}, are given respectively by
\begin{align}
  W_{\tau_b^{(2)} + k}^{(1)} \approx c_2\sqrt{k} \text{ and }
  W_{\tau_b^{(2)} + k} \approx bZ - k,
\label{DYN-RHO-EQ-1}
\end{align}
for $k$ such that $\tau_b^{(2)} + k \leq \tau_{eq}.$ As discussed
previously, fluctuations of order $\sqrt{k}$ arise in workload due to
the Central Limit Theorem, and this phenomenon, as we shall see below,
gains relevance only when $\rho = 1.$ Since our interest here is in
studying delays due to the occurrence of two big jumps, as in Case 2,
if there exists a $K < b(Z-1)$ such that $(K + \tau_b^{(2)})$-th
customer brings a job of size $b - O(\sqrt{b})$ or larger, then at
least one job gets delayed by $b$ units or more. The contribution to
$\Pr\{ \tau_b^{(1)} < \tau_0\}$ due to the occurrence of 2 jumps can
be calculated as below:
\begin{align}
  \label{RHS_PROB_1}
  P_{\textnormal{2 jumps}}(b)
  &:=\Pr\left\{ \tau_b^{(2)} < \tau_0 \right\} \times \Theta \left( \Pr \left\{
    V_{\tau_b^{(2)} + k} > b - \sqrt{b} \text{ for some } k \leq b(Z-1)
    \  \left \vert \frac{}{} \right. \tau_b^{(2)} < \tau_0  \right\}
    \right)\\
  &=\Theta \left( \mathbb{P} \left\{ V>b\right\} \times \sum_{k=1}^{\infty
    }\mathbb{P}\left\{ bZ>k,V_{ k}>b - \sqrt{b} \right\} \right) \nonumber\\
  &=\Theta \left( \mathbb{P}\left\{ V>b\right\} ^{2} \times \sum_{k=1}^{\infty }
    \mathbb{P}\left( bZ>k\right) \right) = \Theta \left( b
    \bar{B}^2(b) \right).  \nonumber
\end{align}
Following the same line of reasoning as in Case 2, we obtain the
following contribution to
$\E_{\bf 0}[\sum_{k=0}^{\tau_0-1} I(W_k^{(1)} > b)]$ due to 2 jumps:
\begin{align}
  \label{RHS-EXP-2JUMPS}
  Q_{\textnormal{2 jumps}}(b) = \Theta \left(b^2 \bar{B}^2(b) \right).
\end{align}

\noindent \textit{Delay due to 1 jump:} Similar to Case 1, when
$\rho = 1,$ the active server accumulates work, albeit at a slower
rate, as given in \eqref{DYN-RHO-EQ-1}.  Since the workload of a
critically loaded single server queue grows like $O(\sqrt{k})$ in $k$
units of time, it is intuitive to expect that if there is a big jump
of size exceeding $b^2$ in the first $O(1)$ units of time of the busy
period, subsequently one of the servers gets blocked for more than
$b^2$ units of time, and the active server which faces all the
incoming traffic accumulates workload of size larger than $b,$ with
non-vanishing probability, in those $b^2$ units of time.  Therefore,
similar to \eqref{SEC-JUMP-PROB}, we have that
\begin{align*}
  \lim_{b \rightarrow \infty} \Pr \left\{ \tau_b^{(1)} < \tau_0 \
  \left \vert \frac{}{} \right. \tau_{b^2}^{(2)} < \tau_0 \right\} > 0.
 \end{align*}
 Therefore, due to \eqref{APPROX-AFTER-JUMP1}, the contribution to
 $\Pr \{ \tau_b^{(1)} < \tau_0 \}$ due to the arrival of only one big
 job is given by
\begin{align*}
  P_{\textnormal{1 jump}}(b) &\approx   \Pr \left\{ \tau_b^{(1)} < \tau_0 \
  \left \vert \frac{}{} \right. \tau_{b^2}^{(2)} < \tau_0 \right\}
                               \times \Pr\left\{ \tau_{b^2} < \tau_0
                               \right\} = \Theta \left( \bar{B}(b^2) \right),
\end{align*}
which is negligible compared to the right hand side of
\eqref{RHS_PROB_1}. As a result, we have that
\begin{align*}
  \Pr \left\{ \tau_b^{(1)} < \tau_0 \right\} \sim P_{\textnormal{2
  jumps}}(b) = \Theta \left( b \bar{B}^2(b) \right).
\end{align*}
However, accounting for the number of jobs that experience at least
$b$ units of delay dramatically changes the contribution of this
single huge jump in the computation of steady-state delay
probabilities. In particular, a single jump of size exceeding $b^2$
blocks one of the servers for $V\ |\ V > b^2 \approx b^2Z$ units of
time, and if we perform calculations similar to Case 1, we shall
obtain that $\Theta(b^2)$ jobs experience delays larger than $b.$ As a
consequence, we have the following contribution in the single, huge
jump regime:
\begin{align*}
  Q_{\textnormal{1 jump}} := \E \left[ \sum_{i=0}^{\tau_0-1} I\left(
  W_k^{(1)} > b\right) \ \left \vert \frac{}{}
  \right. \tau_{b^2}^{(2)} < \tau_0 \right] \times \Pr \left\{
  \tau_b^{(2)} < \tau_0 \right\}
  = \Theta \left( b^2 \Pr \left\{ V > b^2\right\} \right),
\end{align*}
which might not be negligible to the corresponding contribution due to
2 jumps derived in \eqref{RHS-EXP-2JUMPS}. In fact, as demonstrated in
an example in the Introduction, this contribution due to single huge
jump could be larger than its counterpart for 2 jumps based on the
slowly varying function $L(\cdot)$ (consider the example
$L(x) = \log (1+x)$). As a result, we have two competing components in
the expression for steady-state probability of delay in \eqref{FV}.

We conclude with a heuristic explanation of the mechanism involving
one jump for the case $\alpha \in \left( 1,2\right) $ if $\rho =1$. In
this case, once a server is blocked for $k$ units of time, the active
server operates as a critical single-server queue, processing jobs
requiring services with infinite variance, and due to the generalized
Central Limit Theorem, the workload of the critical queue exhibits
fluctuations of order $O( k^{1/\alpha })$. Therefore, if the initial
huge jump, which occurs within $O(1) $ units of time at the beginning
of the busy period, is of size larger than $ b^{\alpha },$ then this
huge job blocks one of the servers for more than $b^\alpha$ units of
time, and as a result, $\Theta(b^\alpha)$ jobs wait for a duration
larger than $b.$ Reasoning as in the finite variance case, the
contribution to steady-state delay due to the arrival of one huge job
is
$\Theta ( b^{\alpha }\mathbb{P}\{ V>b^{\alpha }\}) =\Theta ( b^{\alpha
}b^{-\alpha ^{2}}L( b^{\alpha })).$
On the other hand, the contribution arising from two jumps as in Case
2, namely, according to (\ref{RHO_L_1}), remains
$\Theta ( b^{2-2\alpha }L\left( b\right) ^{2}) $, which is negligible
compared to $\Theta ( b^{\alpha }b^{-\alpha ^{2}}L( b^{\alpha })) $
because $ \alpha \in \left( 1,2\right) $ implies
$2\alpha -2>\alpha ^{2}-\alpha $.  Hence, we arrive at the estimate
(\ref{IV}) in Theorem 1.

\section{Proof of lower bound}
\label{SEC-LB}
\noindent The objective of this section is to prove the following
result.

\begin{proposition}
\label{Proposition_LB}
Suppose that Assumption \ref{ASSUMP_LEFT_TAIL} holds,
and that $\rho =1$. Then, if Assumption \ref{ASSUMP-DIST-V} holds with $
\alpha >2$, there exists $c_{1}>0$ and $b_{0}>0$ such that for all $b>b_{0}.$
\begin{equation*}
\mathbb{P}\left\{ W_{\infty }^{(1)}>b\right\} \geq c_{1}\left( b^{2}\bar{B}
\xspace(b^{2})+b^{2}\bar{B}\xspace^{2}(b)\right) .
\end{equation*}
On the other hand, if $\bar{B}\left( x\right) \sim cx^{-\alpha }$ as $
x\rightarrow \infty $ for some $c>0$ and $\alpha \in \left( 1,2\right) $,
then there exists $c_{1}>0$ and $b_{0}>0$ such that for all $b>b_{0},$
\begin{equation*}
\mathbb{P}\left\{ W_{\infty }^{(1)}>b\right\} \geq c_{1}\left( b^{\alpha }%
\bar{B}\xspace(b^{\alpha })\right) .
\end{equation*}
\end{proposition}
\noindent We now provide the proof of Proposition
\ref{Proposition_LB}. \\

\noindent 
\textbf{Case 1:} (Under the assumption that $\alpha >2$). We first
derive a lower bound based on a single big jump of size exceeding
$b^{2}.$ Let $N_A(t)$ denote the number of jobs that arrive in the
interval $(0,t].$ Let $b>2$ and consider the event, $D_1$, with the
following properties:

\begin{itemize}
\item[1)] The coordinate $W_{1}^{\left( 2\right) }>6b^{2}$ (that is,
  Job 0 blocks one of the servers for $\Omega\left( b^{2}\right) $
  time units).

\item[2)] The total amount of work brought by all the jobs that arrive
  in the time interval $(0,2b^2]$ does not exceed $3b^2.$ In other
  words, $V_1 + \ldots + V_{N_A(2b^2)} \leq 3b^2.$

\item[3)] Every job that arrives in the time interval $[b^2,2b^2]$
  experiences delay for at least $b$ units of time before getting
  processed. That is,
  \[ \min_{N_A \left(b^2 \right) \leq n \leq N_A \left( 2b^2 \right)}
  W_n^{(1)} > b.\]
\end{itemize}
\noindent
On the set $D_1,$ the dynamics of the queue described by recursions
\eqref{MIN-REC} and \eqref{MAX-REC} reduces to
\begin{align*}
  W_n^{(1)} = \left(W_{n-1} + X_n \right)^+ \text{ and } W_n^{(2)} =
  W_{1}^{(2)} - \left( T_2 + \ldots + T_n \right)
\end{align*}
for $2 \leq n \leq N_A(2b^2).$
Further, if we let $S_1 := W_1^{(1)} = 0$ and
$S_n := X_2 + \ldots + X_n$ for $n \geq 2,$ then the following holds
on the set $D_1:$
\begin{align*}
  \label{LB-CASE1-INTER1}
  \min_{N_A \left(b^2 \right) \leq n \leq N_A \left( 2b^2 \right)}
  W_n^{(1)} \geq \min_{N_A \left(b^2 \right) \leq n \leq N_A \left(
  2b^2 \right)} S_n.
\end{align*}
As a result,
\begin{align*}
  &\mathbb{E}_{\bf 0}\left[ \sum_{k=0}^{\tau _{0}-1}I\left( W_{k}^{\left(
    1\right) }>b\right) \right] \geq \mathbb{E}_{\bf 0}\left[
    \sum_{k=0}^{\tau  _{0}-1}I\left( W_{k}^{\left( 1\right) }>b\right)
    ;D_1 \right] \\ 
  &\quad\quad \geq \mathbb{E}_{\bf 0}\left[ N_A\left(2b^2\right) - 
    N_A\left(b^2\right) ; W_{1}^{\left( 2\right)
    }>6b^{2},\min_{N_A\left(b^{2}\right) \leq n\leq
    N_A \left(2b^{2}\right)} S_{n}>1.5b,\ \sum_{i=1}^{N_A(2b^2)} V_i \leq 3b^2
    \right]\\
  &\quad\quad \geq \Pr \left\{ X_1 > 6b^2\right\} \E\left[ N_A\left(2b^2\right) -
    N_A\left(b^2\right); D_1'\right],
\end{align*}
where the event 
\begin{align*}
  D'_1 :=    \left\{ N_A\left(b^2\right) \geq 0.5b^2, \ N_A\left( 2b^2 \right)
  \in \left[ 1.5b^2,2.5b^2\right], \min_{0.5b^2 \leq n \leq 2.5b^2}
  S_n > b, \ \sum_{i=1}^{\lceil 2.5b^2 \rceil} V_i \leq 3b^2 \right\}
\end{align*}
has probability at least
\[\Pr \left\{ \inf_{0.5 \leq t \leq 2.5} \sigma B(t) > 1
\right\}(1-o(1))\]
because of functional CLT and the facts that
$N_{A}\left( x \right) /x \rightarrow 1$ and
$(V_1 + \ldots + V_n)/n \rightarrow 1$ with probability one.  Here
$B(\cdot)$ is a standard Brownian motion and $\sigma^2$ denotes the
variance of $X.$ Additionally, due to the regenerative ratio
representation \eqref{REG-RAT-REP} and the regularly varying nature of
the tail of $X$ (recall that $\Pr\{ X > x\} \sim \bar{B}(x)$ as
$x \rightarrow \infty$), we conclude that
\begin{align}
  \label{LB_1a}
  \Pr\left\{ W_\infty^{(1)} > b\right\}
  &\geq \frac{ b^2 \Pr \left\{ X > 6b^2
    \right\} \times \Pr \left( D_1' \right)}{\E \tau_0}\\
  &\geq c_{1}b^{2}\bar{B}\left( b^{2}\right) \nonumber
\end{align}
for some $c_{1}>0$ and all $b$ large enough.\\

\noindent 
\textbf{Case 2:} (Also under the assumption that $\alpha >2$). We now
derive a lower bound based on the occurrence of two jumps, each of
size exceeding $b.$ Let $b>2$ and consider the event, $D_2,$
 with the following
properties:

\begin{itemize}
\item[1)] The coordinate $W_{1}^{\left( 2\right) }> 5b$ (that is, Job
  0 blocks one of the servers for $\Omega\left( b\right) $ time
  units).

\item[2)] Apart from Job 0, only one of the $N_A(b)$ jobs that arrive
  in the time interval $(0,b]$ bring a service requirement of size
  exceeding $5b.$

\item[3)] The number of customers who arrive during the time intervals
  $(0,b]$ and $(b,2b]$ are numbers between $0.5b$ and $1.5b.$
  Alternatively,
  $N_A(b) \in [0.5b,1.5b] \text{ and } N_A(2b) - N_A(b) \in
  [0.5b,1.5b].$
\end{itemize}
\noindent 
So, on the set $D_2$ we have that at least
$N_A(2b) - N_A(b) \geq 0.5b $ jobs experience a
waiting time more than $b$ units of time, and hence
\begin{align*}
  \mathbb{E}_{\bf 0}\left[ \sum_{k=0}^{\tau _{0}-1}I\left( W_{k}^{\left(
  1\right) }>b\right) ;D_2 \right] \geq 0.5b \Pr\left(D_2\right).
\end{align*}
However, since $N_A(x)/x \rightarrow \infty,$ we have that
\[ \Pr\left\{ N_A(b) \in [0.5b,1.5b],\ N_A(2b)-N_A(b) \in [0.5b,1.5b]
\right\} \sim 1,\]
as $b \rightarrow \infty.$ As a result, 
\begin{align*}
  \Pr\left( D_2 \right) 
  &\geq \left(1-o(1)\right)\sum_{k \leq
    0.5b}\mathbb{P}_{\bf 0}\left\{  
    W_1^{(2)} > 5b,\ V_{k}>5b, \ \bigcap_{i \leq 1.5b, i \neq
    j}\left\{ V_i < 5b \right\}\right\} \\    
  &\geq 0.5b \mathbb{P}\left\{ X_1 >5b\right\} \bar{B}(5b) \left(
    1-\bar{B}(5b)\right)^{1.5b}\left(1-o(1)\right)\\
  &\geq b \bar{B}^2\left(5b\right) \left(1-o(1)\right)
\end{align*}
Then, as in Case 1, due to the regenerative ratio representation
\eqref{REG-RAT-REP} and the regularly varying nature of
$\bar{B}(\cdot),$ we conclude that there exists a constant $c_2$ such
that
\begin{equation}
\Pr\left\{ W_\infty^{(1)} > b\right\} \geq \frac{0.5b \times
  b\bar{B}^2(5b)}{\E \tau_0}\left(1-o(1)\right) \geq
c_{1}b^{2}\bar{B}\left( b\right) ^{2}.  \label{LB_1b} 
\end{equation}
Combining (\ref{LB_1a}) and (\ref{LB_1b}) we obtain the statement of
Proposition \ref{Proposition_LB} for the case $\alpha > 2.$\\

\noindent
\textbf{Case 3: } We now consider the assumption that
$\alpha \in (1,2)$ and $\bar{B}\left( x\right) \sim cx^{-\alpha }$ as
$x\rightarrow \infty.$ The strategy is similar to Case 1. Define an
event, $D_3,$ satisfying the following properties:
\begin{itemize}
\item[1)] The coordinate $W_{1}^{\left( 2\right) }>6b^{\alpha}$ (that
  is, Job 0 blocks one of the servers for
  $\Omega\left( b^{\alpha}\right) $ time units).

\item[2)] The total amount of work brought by all the jobs that arrive
  in the time interval $(0,2b^\alpha]$ does not exceed $3b^\alpha.$ In
  other words, $V_1 + \ldots + V_{N_A(2b^\alpha)} \leq 3b^\alpha.$ 

\item[3)] Every job that arrives in the time interval
  $[b^\alpha,2b^\alpha]$ experiences delay for at least $b$ units of
  time before getting processed. That is,
  \[ \min_{N_A \left(b^2 \right) \leq n \leq N_A \left( 2b^2 \right)}
  W_n^{(1)} > b.\]
\end{itemize}
\noindent
Then, following the same steps as in Case 1, we obtain that
\begin{align*}
  \E_{\bf 0}\left[ \sum_{k=1}^{\tau_0} I\left( W_k^{(1)} >
  b\right)\right] \geq \bar{B}\left( 6b^\alpha \right) \E\left[
  N_A\left(2b^\alpha\right) - N_A\left(b^\alpha\right); D_3'\right]
\end{align*}
where the event 
\[ D_3' := \left\{ N_A\left(b^\alpha\right) \geq 0.5b^\alpha, \
  N_A\left( 2b^\alpha \right) \in \left[
    1.5b^\alpha,2.5b^\alpha\right], \min_{0.5b^\alpha \leq n \leq
    2.5b^\alpha} S_n > b, \ \sum_{i=1}^{\lceil 2.5b^\alpha \rceil} V_i
  \leq 3b^\alpha \right\}\]
has non-vanishing probability as $b \rightarrow \infty$ because
$b^{-1}S_{\left[ tb^{\alpha }\right] }$ converges weakly in
$ D[0,\infty $), to a Stable process $Z\left( \cdot \right).$ As a
result, we obtain
\begin{align*}
  \E_{\bf 0}\left[ \sum_{k=1}^{\tau_0} I\left( W_k^{(1)} >
  b\right)\right] \geq \bar{B}\left( 6b^\alpha \right) \times
  \Pr\left\{ \inf_{1 \leq t \leq 3} Z(t) > 1 \right\} \left(1-o(1)\right).
\end{align*}
This observation, along with the regenerative ratio representation
\eqref{REG-RAT-REP}, concludes the proof of Proposition
\ref{Proposition_LB}.

\section{Proof of upper bound}
\label{SEC-UB1}
\noindent The objective of this section is to prove the following
proposition.

\begin{proposition}
  Suppose that Assumption \ref{ASSUMP_LEFT_TAIL} holds, and that
  $\rho =1$.  Then, if Assumption \ref{ASSUMP-DIST-V} holds with
  $\alpha >2$, there exist $c_{1} > 0$ and $b_{0} > 0$ such that for
  all $b>b_{0},$
\begin{equation*}
\mathbb{P}\left\{ W_{\infty }^{(1)}>b\right\} \leq c_{1}\left( b^{2}\bar{B}%
\xspace(b^{2})+b^{2}\bar{B}\xspace^{2}(b)\right) .
\end{equation*}
On the other hand, if $\bar{B}\left( x\right) \sim cx^{-\alpha },$ as
$ x \rightarrow \infty,$ for some $c>0$ and
$\alpha \in \left( 1,2\right) $, then one can find positive constants
$c_{1}$ and $b_{0}$ such that for all $b>b_{0},$
\begin{equation*}
  \mathbb{P}\left\{ W_{\infty }^{(1)}>b\right\} \leq c_{1}\left( b^{\alpha }
    \bar{B}\xspace(b^{\alpha })\right) .
\end{equation*}
\label{PROP-UB}
\end{proposition}

The rest of this section is devoted to the proof of Proposition
\ref{PROP-UB}.  First, pick $\delta _{-}, \delta, \delta _{+}$ such
that $0<\delta _{-} < \delta < \delta _{+}<1.$ In addition to the
stopping times
\begin{equation*}
  \tau _{x}^{\left( i\right) }=\inf \left\{n\geq 0:W_{n}^{\left( i\right)
    }>x \right\}, 
\end{equation*}
which are defined for $x > 0, i=1 \text{ and } 2,$ let us define
\begin{equation*}
  \bar{\tau}_{b\delta _{+}}^{\left( 2\right) }=\inf \left\{n\geq \tau _{b\delta
      _{-}}^{\left( 2\right) }:W_{n}^{\left( 2\right) }\leq b\delta _{+}\right\}.
\end{equation*}
Additionally, let
\begin{align*}
  B_{1}\left( b\right) &:=\mathbb{E}_{\bf 0}\left[ \sum_{k=0}^{\tau
                         _{0}-1}I\left( W_{k}^{\left( 1\right) }>b\right) I\left( \bar{\tau}_{b\delta
                         _{+}}^{\left( 2\right) } > \tau _{b\delta }^{\left( 1\right) }\right)
                         \right] \text{ and }\\
  B_{2}\left( b\right) &:=\mathbb{E}_{\bf 0}\left[ \sum_{k=0}^{\tau
                         _{0}-1}I\left( W_{k}^{\left( 1\right)}>b\right) I\left( \bar{\tau}_{b\delta
                         _{+}}^{\left( 2\right) } \leq \tau _{b\delta}^{\left( 1\right) }\right) \right].
\end{align*}
Then, it follows from the regenerative ratio representation
\eqref{REG-RAT-REP} that
\begin{align}
  \Pr \left\{ W_{\infty}^{(1)} > b \right\} = \frac{B_1(b) + B_2(b)}{\E_{\bf
  0}[\tau_0]}.
  \label{RR-INTERMSOF_B}
\end{align}
The term $B_1(b)$ corresponds to the case where all the actions
happen: once there is a large jump in $W^{(2)}$ which takes it beyond
$b \delta_-,$ one of the servers gets blocked for a long time, and the
other server which faces the entire traffic in that duration piles up
work more than $b \delta$. On the other hand, the term $B_{2}( b) $
corresponds to the case where the first jump is wasted: that is, there
is not enough buildup in $W^{(1)}$ after the occurrence of first jump
in $W^{(2)}.$ The rigorous procedure of obtaining upper bounds for
$B_1(b)$ and $B_2(b)$ is divided into several parts:

\bigskip

Part 1) First, we obtain an upper bound for $\E_{\bf w} [\tau_0]$
uniformly over all initial conditions ${\bf w} = (w_1,w_2).$ This
shall be useful in obtaining upper bounds for both $B_1(b)$ and
$B_2(b)$ because of the simple observation that
\[ \E_{\bf w}\left[ \sum_{k=0}^{\tau_0-1} I \left( W_k^{(1)} >
    b\right)\right] \leq \E_{\bf w}[\tau_0].\]
Additionally, in an attempt to obtain a stochastic description of the
workload $W^{(2)}$ after it exceeds $\delta b_-,$ we derive a
stochastic domination result for $W_{\tau_{b}^{(2)}}^{(2)}$ which
shall be useful.

\bigskip

Part 2) We reduce the contribution of the first term $B_1(b)$ into a
large deviations problem for zero-mean random walks with regularly
varying increments. We use the stochastic domination result obtained
in Part 1) along with another domination argument, in terms of a
suitably defined critically loaded single-server queue, to account for
all of what happens after the first jump. In turn, the introduction of
the single-server queue sets the stage for the use of uniform large
deviations for random walks. The analysis of part 2) emphasizes the
convenience of partitioning the numerator in \eqref{REG-RAT-REP} into
$B_1(b)$ and $B_2(b).$

\bigskip 

Part 3.a)\ This is the portion of the argument that requires
$\alpha >2$. It invokes classical results for uniform large deviations
of regularly varying random walks available due to Nagaev (uniform in
the sense that the asymptotics jointly account both the Brownian
approximations in the CLT scaling regime and the large deviations
approximations in scaling regimes beyond that of CLT). The execution
of part 2) involves routine estimations of one dimensional integrals
using basic properties of regularly varying distributions. We obtain
the required upper bound for $B_1(b)$ after some elementary
simplifications.

\bigskip

Part 3.b) The analysis here is entirely parallel to that of part 3.a),
except that the uniform estimates involve an approximation using an
$\alpha$-stable process (instead of Brownian motion as in Part 3.a)).

\bigskip

Part 4) is devoted to obtaining an upper bound for the residual term
$B_2(b).$ This is accomplished by first performing calculations that
result in an intermediate bound for $B_2(b)$ in terms of expected
number of jobs that wait for duration longer than $b$ after the first
jump. The second calculation involves obtaining a good upper bound for
$\Pr_{\bf w} \{\tau_b^{(2)} < \tau_0\}$ uniformly over initial
conditions ${\bf w} \in \{ (w_1,w_2): w_2 < b\delta_+\}.$

In the following subsections we shall estimate the contributions of
$B_1(b)$ and $B_2(b)$ following the outline presented above. In order
to streamline the presentation, we present proofs of some of the
results in the appendix.

\subsection*
{Part 1) Some useful upper bounds} 
Recall our earlier definition $X := V - T.$ As mentioned previously,
the goal of this subsection is to provide some generic bounds which
will be useful in deriving upper bounds for both $B_1(b)$ and
$B_2(b).$
\begin{lemma}
  \label{Lem_LB_ET0}
  Suppose that $\rho = 1.$ Then there exist positive constants $C_{1}$
  and $C_{0}$ such that for all ${\bf w} =(w_1,w_2)$ satisfying
  $0 \leq w_1 \leq w_2,$
\begin{equation*}
  \E_{\bf w}\left[ \tau _{0}\right] \leq C_{1}w_{2}+C_{0}.
\end{equation*}
\end{lemma}

\begin{remark}
  \textnormal{The conclusion of Lemma \ref{Lem_LB_ET0} holds true for
    every $\rho < 2.$ Our proof for Lemma \ref{Lem_LB_ET0} can be
    easily modified to accommodate every $\rho < 2.$}
\end{remark}

\begin{lemma}
\label{LEM_ST_DOM}
For every $x\geq b$ and ${\bf w}=\left( w_{1},w_{2}\right)$ with
$0 \leq w_1 \leq w_{2}<b,$
\begin{equation*}
  \Pr_{\bf w} \left\{ W_{\tau _{b}^{\left( 2\right) }}^{\left(
        2\right) }>x\ \left \vert \frac{}{} \right. \ \tau_{b}^{\left(
        2\right) }<\tau _{0} 
  \right\}\leq \Pr\left\{ X + b > x \ |\ X>b\right\}. 
\end{equation*}
In other words, $W_{\tau _{b}^{\left( 2\right) }}^{\left(2\right)}$
given $\tau _{b}^{\left( 2\right) }<\tau _{0}$ is stochastically
dominated by $X+b$ given $X > b.$
\end{lemma}

If $\rho < 2,$ it is intuitive to expect the servers to effectively
drain work whenever $W^{(2)}$ is large. Lemma \ref{LEM-LYAP-ET0},
whose proof is given in Appendix \ref{APP-OTHERS}, asserts the same
when $\rho = 1.$
\begin{lemma}
  \label{LEM-LYAP-ET0}
  There exist positive constants $C$ and $\varepsilon$ such that
  \begin{equation*}
    \E_{(w_1,w_2)}\left[ W_1^{(1)} + W_1^{(2)} \right] <
    (w_1+w_2)-\varepsilon 
  \end{equation*}
  as long as $w_{2}\geq C.$
\end{lemma}

\noindent Lemma \ref{Lem_LB_ET0} follows as a corollary of Lemma
\ref{LEM-LYAP-ET0} via a standard Lyapunov argument.

\begin{proof}[Proof of Lemma \ref{Lem_LB_ET0}]
  Let $A := \inf\{ (w_1,w_2): w_1 \leq w_2 \leq C\}$ and
  $T_A := \inf\{ n \geq 1 : W_n \in A\}.$ Additionally, let
  $V \left((w_1,w_2)\right) = {(w_1 + w_2)}/{\ve}$ for
  $0 \leq w_1 \leq w_2.$ Here $C$ and $\ve$ are chosen as in Lemma
  \ref{LEM-LYAP-ET0}. It follows from recursions \eqref{MIN-REC} and
  \eqref{MAX-REC} that
  \[ \sup_{(w_1,w_2) \in A}\E_{(w_1,w_2)} \left[ V
    \left(W_1^{(1)},W_2^{(2)}\right)\right] \leq \frac{\E \left[
      \left(C + V - T\right)^+ + \left(C - T\right)^+\right]}{\ve} =:
  C_2 < \infty.\]
  This observation, in conjunction with Lemma \ref{LEM-LYAP-ET0} and
  Theorem 11.3.4 of \cite{MR1287609}, results in
  \begin{align}
  \label{INTER-LYAP}
    \E_{(w_1,w_2)}[T_A] \leq \frac{w_1+w_2}{\epsilon} + C_2 \leq
    \frac{2}{\epsilon}w_2 + C_2
  \end{align}
  for every $0 \leq w_1 \leq w_2.$ Moreover, since
  $\inf_{{\bf w} \in A} \Pr_{\bf w}\{ W_1^{(2)} = 0\} \geq \Pr\{ T >
  C\} > 0,$
  it follows from a simple geometric trials argument that
  $\sup_{{\bf w} \in A}\E_{\bf w}[\tau_0] < \infty.$ This observation,
  along with \eqref{INTER-LYAP}, proves the claim.
\end{proof}

\begin{proof}[Proof of Lemma \ref{LEM_ST_DOM}]
Note that 
\begin{align*}
  \mathbb{P}_{\bf w}\left\{ W_{\tau _{b}^{\left( 2\right) }}^{\left( 2\right)
  }>x,\tau _{b}^{\left( 2\right) }<\tau _{0}\right\}
  =\sum_{k=1}^{\infty }\mathbb{P}_{\bf w}\left\{ W_{k}^{\left( 2\right) }>x,\tau
  _{0}>k,\tau _{b}^{\left( 2\right) }=k\right\}.
  \end{align*}
  If $\tau_b^{(2)} = k,$ it follows from recursion \eqref{MAX-REC}
  that $W_k^{(2)} = W_{k-1}^{(1)} + X_k.$ Here recall that
  $X_k = V_{k-1} - T_{k}.$ Therefore,
  \begin{align}
    \label{INTER-ST-DOM}
    \mathbb{P}_{\bf w}\left\{ W_{\tau _{b}^{\left( 2\right) }}^{\left( 2\right)
    }>x,\tau _{b}^{\left( 2\right) }<\tau _{0}\right\}
    &=\sum_{k=1}^{\infty } \Pr_{\bf w}\left\{ X_{k}>x-W_{k-1}^{\left(
      1\right) },\tau _{0}>k-1,\tau _{b}^{\left( 2\right)
      }>k-1\right\}\\
    &=\sum_{k=1}^{\infty }\mathbb{E}_{\bf w}\left[ I\left( \tau _{0}>k-1,\tau
      _{b}^{\left( 2\right) }>k-1\right) \bar{F}\left( x-W_{k-1}^{\left( 1\right)
      }\right) \right] \nonumber
\end{align}
Observe that $W_{k-1}^{(2)} < b$ whenever $\tau_b^{(2)} > k-1.$
Additionally, since $x$ is taken to be larger than $b,$
\[ \frac{\bar{F}\left(x - W_{k-1}^{(1)} \right)}{\bar{F}\left(b -
    W_{k-1}^{(1)} \right)} \leq \frac{\bar{F}(x-b)}{\bar{F}(b)} \wedge
1 = \Pr \left\{ X + b > x \ |\ X > b\right\}\]
on the set $\{\tau_b^{(2)} > k-1\}.$ Therefore,
\begin{align*}
  &\Pr_{\bf w}\left\{ W_{\tau _{b}^{\left( 2\right) }}^{\left( 2\right) }>x,\tau
    _{b}^{\left( 2\right) }<\tau _{0}\right\} \\
  &\quad\quad \leq \Pr \left\{ X + b > x \ |\ X > b\right\} \times
    \sum_{k=1}^{\infty }\E_{\bf w}\left[ I\left( \tau _{0}>k-1,\tau _{b}^{\left( 2\right) }>k-1\right) \bar{F}\left(
    b-W_{k-1}^{\left( 1\right) }\right) \right] \\
  &\quad\quad=\Pr\left\{ X+b>x\ |X>b\right\} \Pr_{\bf w}\left\{ \tau _{b}^{\left( 2\right) }<\tau
    _{0}\right\} ,
\end{align*}
where the last expression was obtained by letting $x=b$ in the second
line in \eqref{INTER-ST-DOM}. The last inequality is equivalent to the
statement of Lemma \ref{LEM_ST_DOM}, and this concludes the proof.
\end{proof}

\subsection*{Part 2) Reduction to a zero-mean random walk problem} 
Recall our earlier definition $X_{n} := V_{n-1}-T_{n}$ for $n \geq 1,$
where $(V_n: n \geq 1)$ are i.i.d. copies of $V$ and $(T_n : n \geq 1)$
are i.i.d. copies of $T.$ Additionally, we had set $V_0 := 0.$
Further, define $S_0 := 0, \ S_n := X_1 + \ldots + X_n,$ and
\[ N_A(t) := \sup \left\{ n \geq 0: T_1 + \ldots + T_n \leq t\right\}
\vee 0\]
for $t \geq 0.$ Here we follow the usual convention that
$\sup \emptyset= -\infty.$ Therefore, $N_A(0) = 0.$ Note that $N_A(t)$
is the number of customers that arrive in the time interval $(0,t].$
In addition to the above definitions, let $X := V-T$ and define
  \begin{equation*}
    B_{3}\left( b\right) :=\mathbb{E}\left[ I\left( \max_{0\leq n\leq N_{A}\left(
            X\right) + 1}2\left\vert S_{n}\right\vert > \left( \delta -\delta _{-}\right)
        b\right) \left( X+\max_{0\leq n\leq N_{A}\left( X\right) + 1}\left\vert
          S_{n}\right\vert \right) \ \left|\frac{}{}\right. X > b \delta_+ \right].
  \end{equation*}
Our objective in this subsection is to show the following result.
\begin{lemma}
  \label{LEM_RED_RW}
  Suppose that Assumptions \ref{ASSUMP-DIST-V} and
  \ref{ASSUMP_LEFT_TAIL} hold, and that $\rho = 1.$ 
  Then,
  \begin{equation*}
    B_{1}\left( b\right) =O\left( \Pr_{\bf 0}
      \left\{ \tau_{b \delta_+}^{(2)} < \tau_0 \right\} \times B_3(b) \right).
  \end{equation*}
\end{lemma}

\noindent Let $\mathcal{F}_n$ denote the $\sigma-$algebra generated by
the random variables $V_k$ and $T_k, k \leq n.$ Then 
\begin{align*}
  B_1(b)  &= \E_{\bf 0} \left[ I\left( \bar{\tau}_{b\delta_+}^{(2)} >
            \tau_{b \delta}^{(1)}\right) \E_{\bf 0} \left[ \sum_{k=0}^{\tau_0-1}
            I\left( W_k^{(1)} >b\right)  \ \left| \frac{}{}
            \right. \mathcal{F}_{\tau_{b\delta}^{(1)}}\right] \right].
\end{align*}
Since $W_k^{(1)}$ is smaller than $b$ for
$k < \tau_{b \delta}^{(1)},$ on the set
$\{\tau_{b}^{(1)} < \tau_0\},$ we have
\[ \E_{\bf 0} \left[ \sum_{k=0}^{\tau_0-1} I\left( W_k^{(1)} >b\right)
  \ \left| \frac{}{} \right. \mathcal{F}_{\tau_{b\delta}^{(1)}}\right]
= \E_{{\bf W}_{\tau_{b \delta}^{(1)}}} \left[ \sum_{k=0}^{\tau_0-1}
  I\left( W_k^{(1)} >b\right) \right] \leq \E_{{\bf
           W}_{\tau_{b \delta}^{(1)}}}  \left[ \tau_0 \right]. \]
Then, due to Lemma \ref{Lem_LB_ET0},
\begin{align*}
  B_1(b)   &\leq \E_{\bf 0} \left[ I\left(\bar{\tau}_{b\delta_+}^{(2)}
             > \tau_{b \delta}^{(1)}, \tau_{b}^{(1)} < \tau_0 \right) \left( C_1 
             W_{\tau_{b \delta }^{(1)}}^{(2)} + C_0 \right) \right].
\end{align*}
\noindent First, observe that whenever
$\bar{\tau}_{b\delta_+}^{(2)} > \tau_{b \delta}^{(1)},$ we must also
have that $\tau_{b\delta_+}^{(2)} = \tau_{b\delta_{-}}^{(2)}.$
Otherwise, from the definition of $\bar{\tau}_{b\delta_+}^{(2)},$ it
follows that $\bar{\tau}_{b\delta_+}^{(2)} = \tau_{b\delta_{-}}^{(2)}$
which in turn occurs earlier than $\tau_{b\delta}^{(1)},$ and this
contradicts our blanket assumption
$\bar{\tau}_{b\delta_+}^{(2)} > \tau_{b \delta}^{(1)}.$ Therefore,
\begin{align*}
  B_1(b)   &\leq \E_{\bf 0} \left[ I\left(\bar{\tau}_{b\delta_+}^{(2)}
             > \tau_{b \delta}^{(1)}, \tau_{b\delta_+}^{(2)} =
             \tau_{b\delta_{-}}^{(2)}, \tau_{b}^{(1)} < \tau_0 \right) \left( C_1 
             W_{\tau_{b \delta }^{(1)}}^{(2)} + C_0 \right) \right]\\
           &\leq \E_{\bf 0} \left[ I\left(\tau_{b\delta_+}^{(2)} \leq 
             \tau_{b\delta_{-}}^{(1)} \wedge \tau_0 \right) \E_{\bf 0}
             \left[\left( C_1 
             W_{\tau_{b \delta }^{(1)}}^{(2)} + C_0 \right) I\left(
             \bar{\tau}_{b\delta_+}^{(2)} 
             > \tau_{b \delta}^{(1)} \right) \ \left| \frac{}{}
             \right. \mathcal{F}_{\tau_{b\delta_{+}}^{(2)}}\right]
             \right]. 
\end{align*}
As a consequence of strong Markov property of ${\bf W},$ we have that 
\begin{align*}
  B_1(b) \leq \E_{\bf 0} \left[ I\left(\tau_{b\delta_+}^{(2)} \leq 
  \tau_{b\delta_{-}}^{(1)} \wedge \tau_0 \right) \E_{{\bf W}_{\tau_{b
  \delta_+}^{(2)}}}  \left[\left( C_1 
  W_{\tau_{b \delta }^{(1)}}^{(2)} + C_0 \right) I\left(
  \bar{\tau}_{b\delta_+}^{(2)} 
  > \tau_{b \delta}^{(1)} \right) \right]  \right]
\end{align*}
If we set
$\xi_1 := W_{\tau_{b \delta_+}^{(2)}}^{(1)} \text{ and } \xi_2 := W_{\tau_{b
    \delta_+}^{(2)}}^{(2)},$ again due to the Markov property of ${\bf W},$
\begin{align}
  \label{INTER-B1}
  B_1(b) \leq \E \left[ I( \xi_1 < b\delta_-)
  \E_{(\xi_1,\xi_2)}\left[ \left( C_1 
  W_{\tau_{b \delta}^{(1)}}^{(2)} + C_0 \right) I\left(
  \bar{\tau}_{b\delta_+}^{(2)} 
  > \tau_{b \delta}^{(1)} \right) \right]\right]  \Pr_{\bf 0} \left\{
  \tau_{b\delta_{+}}^{(2)} < \tau_0 \right\},
\end{align}
where $\xi_2,$ by definition, is larger than $b\delta_+.$
\subsection*{Evaluation of the inner expectation}
We analyse the inner expectation
\begin{align*}
  \chi \left(\xi_1,\xi_2 \right) := \E_{(\xi_1,\xi_2)}\left[ \left( C_1 
  W_{\tau_{b \delta}^{(1)}}^{(2)} + C_0 \right) I\left(
  \bar{\tau}_{b\delta_+}^{(2)} 
  > \tau_{b \delta}^{(1)} \right) \right]
\end{align*}
 in \eqref{INTER-B1} by restarting the
queuing system with initial conditions ${\bf W}_0 = (\xi_1,\xi_2).$
Whenever $\bar{\tau}_{b\delta_+}^{(2)} > \tau_{b \delta}^{(1)},$ due
to recursions \eqref{MIN-REC} and \eqref{MAX-REC}, the dynamics of the
queue until $\tau_{b \delta}^{(1)}$ is described by
\begin{align*}
  W_{n}^{(1)} = \left( W_{n-1}^{(1)} + X_{n} \right)^{+} \text{ and }
  W_{n}^{(2)} = W_{n-1}^{(2)} - T_{n}
\end{align*}
for $1 \leq n < \tau_{b \delta}^{(1)},$ in conjunction with
$W_0^{(1)} = \xi_1 < b \delta_-$ and $W_0^{(2)} = \xi_2 > b \delta_+.$
As a result,
\[T_1 + \ldots + T_{\tau_{b \delta}^{(1)}-1} = \xi_2-W_{\tau_{b
    \delta}^{(1)}-1}^{(2)} \leq \xi_2- b \delta_+\]
whenever $\bar{\tau}_{b\delta^+}^{(2)} > \tau_{b \delta}^{(1)}.$
Therefore,
$\tau_{b \delta}^{(1)} \leq N_A\left( \xi_2 - b\delta^+\right) + 1,$
which in turn implies that
\begin{align*}
  \max_{0 < n \leq N_A\left( \xi_2 - b\delta^+\right) + 1} W_n^{(1)}
  > b \delta.
\end{align*}
Consequently,
\begin{align*}
  \label{INTER-INNER-EXP-UB}
  \chi \left( \xi_1, \xi_2\right) \leq \E_{(\xi_1,\xi_2)}\left[ \left(
  C_1 \max_{0 < n \leq N_A\left( \xi_2 - 
  b\delta^+\right) + 1} W_n^{(2)} + C_0\right) I \left( \max_{0 < n
  \leq N_A\left( \xi_2 - b\delta^+\right) + 1} W_n^{(1)} 
  > b \delta \right) \right].
\end{align*}
The following result is verified in Appendix \ref{APP-OTHERS}.
\begin{lemma}
  \label{LEM_SINGLE_Q_DOM}
  Suppose that ${\bf W}_{0}=(w_{1},w_2),$ and recall the definitions
  $X_n := V_{n-1} - T_{n}, \ S_0 := 0$ and
  $S_n := X_1 + \ldots + X_n.$ Then, for all $n\geq 0$,
  \begin{equation*}
    \max_{0 < k \leq n} W_{k}^{\left( i\right) } \leq 2\max_{0 \leq
      k \leq  n}\left\vert S_{k}\right\vert +w_i, \ i=1,2.
  \end{equation*}
\end{lemma}
\noindent As a consequence of Lemma \ref{LEM_SINGLE_Q_DOM},
\begin{align*}
  \chi \left( \xi_1, \xi_2 \right) \leq 
  C_3 \E \left[  \left( \max_{0 \leq n \leq N_A\left( \xi_2 - 
  b\delta^+\right) + 1} 2S_n+ \xi_2\right) I \left( \max_{0 \leq n
  \leq N_A\left( \xi_2 - b\delta^+\right) + 1} 2S_n + \xi_1 
  > b \delta \right) \ \left| \frac{}{} \right. \ \xi_1,
  \xi_2\right]. 
\end{align*}
where the constant $C_0$ been absorbed in another suitable constant
$C_3.$ Then it is immediate from \eqref{INTER-B1} that
\begin{align*}
  B_1(b) \leq C_3 
  \E \left[  \left( \max_{0 \leq n \leq
  N_A\left( \xi_2 -  
  b\delta^+\right) + 1} 2S_n+ \xi_2\right) I \left( \max_{0 \leq n
  \leq N_A\left( \xi_2 - b\delta^+\right) + 1} 2S_n  
  >  \left(\delta - \delta_-\right) b \right) \right]
  \Pr_{\bf 0}  \left\{ \tau_{b \delta_+}^{(2)} < 
  \tau_0\right\}. 
\end{align*}
Here, recall that $\xi_2 := W_{\tau_{b \delta_+}^{(2)}}^{(2)},$ which
is stochastically dominated by the conditional distribution of
$X + b \delta_+$ given that $X > b \delta_+$ (due to Lemma
\ref{LEM_ST_DOM}). Since $B_2(b)$ is a non-decreasing function of
$\xi_2,$ we use the above stochastic dominance to yield
\begin{align*}
  B_1(b) &\leq C_3 \Pr_{\bf 0}  \left\{ \tau_{b \delta_+}^{(2)} < 
           \tau_0\right\} \times \\
         &\quad\quad 
           \E \left[  \left( \max_{0 \leq n \leq
           N_A\left( X \right) + 1} 2S_n+ X + b\delta_+\right) I \left( \max_{0
           \leq n  \leq N_A\left( X \right) + 1} 2S_n  
           >  \left(\delta - \delta_-\right) b \right) \
           \left|\frac{}{}\right. X > b\delta_+ \right].  
\end{align*}
Lemma \ref{LEM_RED_RW} follows from the above inequality once we
observe that $X + b \delta_+ \leq 2 X$ when $X > b\delta_+.$ This
completes the proof of Lemma \ref{LEM_RED_RW}. \hfill{$\Box$}

\subsection*{Part 3.a) Simplifications using uniform large deviations:
the $\alpha > 2$ case}
Using classical results borrowed from the literature on large
deviations for zero-mean random walks, we aim to prove the following
result in this subsection.

\begin{lemma}
\label{Lem_Simplifcation}
Suppose that Assumptions \ref{ASSUMP-DIST-V} and
\ref{ASSUMP_LEFT_TAIL} are in force, $\alpha >2$ and $\rho =1$. Then,
\begin{equation*}
  B_{3}\left( b\right) =O\left( b^{2}\bar{B}\left( b\right) +b^{2} \frac{\bar{B}
      \left( b^{2}\right)}{\bar{B}(b)} \right).
\end{equation*}
\end{lemma}

We begin by recalling results on uniform large deviations for
regularly varying random walks. For example, the following large
deviations result which holds under Assumptions \ref{ASSUMP-DIST-V}
and \ref{ASSUMP_LEFT_TAIL} assuming that $\alpha >2$, is well-known
\begin{equation}
\Pr \{S_{m}>b\}=\left( \bar{\Phi}\left( \frac{b}{\sqrt{m}\sigma }\right)
+m\Pr \{X_{1}>b\}\right) \left( 1+o\left( 1\right) \right) ,\text{ as }%
m\rightarrow \infty ,  \label{SUM-ASYMP1}
\end{equation}%
uniformly for $b>\sqrt{m}$, where $\bar{\Phi}\left( \cdot \right) $ is
the tail of a standard normal distribution. The asymptotic
approximation (\ref {SUM-ASYMP1}) is due to A. V. Nagaev (see Theorem
1.9 of \cite{nagaev1979} or Corollary 7 of
\cite{rozovskii1989probabilities}).

For our purposes, we need an extension of (\ref{SUM-ASYMP1}), in which
$S_{n}$ is replaced by
$\max_{0\leq k\leq m}\left\vert S_{k}\right\vert $. However, we do not
need exact asymptotic results as in (\ref{SUM-ASYMP1}), but only an
asymptotic upper bound. This is the content of the following result,
which is proved in Appendix \ref{APP-OTHERS} as an immediate
consequence of Corollary 1 of \cite{doi:10.1137/1126006}. (For related
uniform sample path large deviations results see \cite{MR2424161}, and
the related Theorem 5 of \cite{doi:10.1137/S0040585X97978877}.)

\begin{lemma}
  \label{LEM-MAX-ABS-SUM-ASYMP}
  Suppose that $V$ satisfies Assumption \ref{ASSUMP-DIST-V} with
  $\alpha >2 $, and $T$ satisfies Assumption
  \ref{ASSUMP_LEFT_TAIL}. Recall that $X_1,X_2,\ldots$ are
  i.i.d. copies of $X = V -T.$ Then, there exists a positive integer
  $m_{0}$ such that for all $x \geq m^{1/2}$ and $m > m_0$
\begin{equation*}
  \Pr \left\{ \max_{0\leq k\leq m}\left\vert S_{k}\right\vert >x\right\}
  \leq 3\left( \Pr\left\{ \max_{0\leq t\leq 1}\sigma \left\vert B\left(
          t\right) \right\vert >\frac{x}{m^{1/2}}\right\} +m\Pr\left\{
      \left\vert X\right\vert >x\right\} \right) ,
\end{equation*}
where $\sigma ^{2}=Var[ X] $ and $B\left( \cdot \right) $ is a
standard Brownian motion.
\end{lemma}

\noindent We also need the following pair of standard results on
regular variation: Karamata's theorem (refer Theorem 1 in Chapter
VIII.9 of \cite {feller1971introduction}) and Potter's bounds (see,
for example, Theorem 1.1.4 of \cite{MR2424161})
\begin{proposition}[Karamata's theorem]
  Suppose that $v(t) = t^{-\alpha}L(t)$ for some slowly varying
  function $L(\cdot)$ and $\alpha$ satisfying $\alpha - \beta > 1.$
  Then
\begin{equation}
  \int_x^\infty u^\beta v(u)du \sim
  \frac{x^{\beta+1}v(x)}{\alpha-\beta-1}, \text{ as } x \rightarrow
  \infty.
  \label{KARAMATA}
\end{equation}
On the other hand, if $\alpha - \beta < 1,$ then 
\begin{equation}
  \label{KARAMATA-II}
  \int_0^x u^\beta v(u)du \sim
  \frac{x^{\beta+1}v(x)}{1-\alpha+\beta}, \text{ as } x \rightarrow
  \infty.
\end{equation}
\end{proposition}

\begin{proposition}[Potter's bounds]
If $v(t)=t^{-\alpha}L\left( t\right) $ for some $\alpha >0$ and some slowly
varying function $L\left( \cdot \right) $ then, for any $\varepsilon \in
\left( 0,\min (\alpha ,1)\right) ,$ there exists a $t_{\varepsilon }>0$ such
that for all $t$ and $c$ satisfying $t\geq t_{\varepsilon }$ and $ct\geq
t_{\varepsilon },$ 
\begin{equation}
(1-\varepsilon )\min \{c^{-\alpha +\varepsilon },c^{-\alpha -\varepsilon
}\}\leq \frac{v(ct)}{v(t)}\leq (1+\varepsilon )\max \{c^{-\alpha
+\varepsilon },c^{-\alpha -\varepsilon }\}.  
\label{POTT_BND}
\end{equation}
\end{proposition}

\noindent We establish Lemma \ref{Lem_Simplifcation} in two parts. The
first task involves analysing the relatively easier term, which has
the running maximum appearing only in the indicator
function. 



\begin{lemma}
  \label{Lem_No_Run_Max_Mult}
  Under Assumption \ref{ASSUMP-DIST-V} with $ \alpha >2$, and
  Assumption \ref{ASSUMP_LEFT_TAIL},
  \begin{equation*}
    \mathbb{E}\left[ I\left( \max_{0\leq n\leq N_A(X) + 1}2\left\vert S_{n}\right\vert
        >\left( \delta -\delta _{-}\right) b\right) X \ \left|\frac{}{}\right.
      X > b\delta_+ \right]
    =O\left( b^{2}\bar{B}\left( b\right)  +  b^{2}\frac{\bar{B}\left(
          b^{2}\right)}{\bar{B}(b)} \right).
  \end{equation*}
\end{lemma}

\noindent Following this, we estimate the term in which the running
maximum appears both multiplying and inside the indicator.

\begin{lemma}
  \label{Lem_Run_Max_Appers}
  Under Assumption \ref{ASSUMP-DIST-V} with $ \alpha >2$, and
  Assumption \ref{ASSUMP_LEFT_TAIL},
  \begin{equation*}
    \mathbb{E}\left[ I\left( \max_{0\leq n\leq N_A(X) + 1}2\left\vert S_{n}\right\vert
        >\left( \delta -\delta _{-}\right) b\right) \max_{0\leq n\leq
        N_A(X) + 1}\left\vert
        S_{n}\right\vert \ \left|\frac{}{}\right. X>b\delta_+ \right]
    =O\left( b^{2}\bar{B}\left( b\right) \right).
\end{equation*}
\end{lemma}
\noindent Lemma \ref{LEM-BSQ-SIMP-TERM}, whose proof is given in
Appendix \ref{APP-OTHERS}, will be useful in proving Lemmas
\ref{Lem_No_Run_Max_Mult} and \ref{Lem_Run_Max_Appers}.
\begin{lemma}
  \label{LEM-BSQ-SIMP-TERM}
  If $v(x) = x^{-\alpha}l(x)$ for some $\alpha > 2$ and a function
  $l(\cdot)$ slowly varying at infinity, then for every $c > 0,$
  \[ \int_{b}^\infty v(t) \exp\left( -c\frac{b^2}{t}\right) dt = O \left(
    b^2 v \left(b^2\right)\right).\]
\end{lemma}

\begin{proof}[Proof of Lemma \ref{Lem_No_Run_Max_Mult}] 
  Letting $c = (\delta-\delta_-)/2,$ observe that
  \begin{align}
    &\mathbb{E}\left[ I\left( \max_{0\leq n\leq N_A(X)}2\left\vert S_{n}\right\vert
      >\left( \delta -\delta _{-}\right) b, \  N_A(X) + 1 \leq 2X\right) X \
      \left|\frac{}{}\right.  X > b\delta_+ \right] \nonumber\\
    &\quad\quad \leq \int_{b\delta_+}^\infty t \Pr \left\{ \max_{0 \leq n
      \leq 2t} \left \vert S_n\right\vert > cb\right\} \frac{\Pr \left\{ X
      \in dt\right\}}{\Pr\left\{ X > b\delta_+\right\}} \nonumber\\ 
    &\quad\quad \leq 3\int_{b\delta_+}^\infty t \Pr \left\{ \max_{0 \leq s
      \leq 1} \sigma B(s) > \frac{cb}{\sqrt{2t}}  \right\} \frac{\Pr\{ X \in
      dt\}}{\Pr\left\{ X > b\delta_+\right\}} + 
      3\int_{b\delta_+}^{\frac{c^2b^2}{2}} 2t^2 \Pr \left\{ 
      \left\vert X  \right\vert> cb \right\} \frac{\Pr \left\{ X \in
      dt\right\}}{\Pr\left\{ X > b\delta_+\right\}} \label{INTER-B31-1} 
  \end{align}
  because of the application of the uniform asymptotic presented in
  Lemma \ref{LEM-MAX-ABS-SUM-ASYMP} in the region $2t \leq c^2b^2$ and
  Central Limit Theorem in the region $2t > c^2b^2.$ Recall from
  \eqref{tail_equiv_F_B} that $\Pr\{ X > x\} \sim \bar{B}(x),$ and
  subsequently, due to Karamata's theorem (\ref{KARAMATA}), we obtain
  \begin{equation*}
    \int_{b\delta_+}^{\infty }t^{2}\mathbb{P}\left\{ X\in dt\right\}
    \leq \mathbb{E} \left[ X^{2}I\left( X> b\delta_+\right) \right]
    =O\left( \int_{b\delta_+}^{\infty }s\mathbb{ 
        P}\left\{ X>s\right\} ds\right) =O\left( b^{2}\bar{B}\left(
        b\right) \right) 
  \end{equation*}
  and therefore
  \begin{equation}
    \frac{\mathbb{P}\left( \left\vert X \right\vert >c
        b\right)}{\Pr\left\{ X > b \delta_+\right\}} 
    \int_{b}^{\infty }t^{2}\mathbb{P}\left\{ X\in 
     dt\right\} =O\left( b^{2}\bar{B}\left( b\right)\right) .
    \label{INTER-B31-2}
  \end{equation}
  To deal with the first term in \eqref{INTER-B31-1}, we do
  integration by parts (by taking
  $u = \Pr\{ \max_{0 \leq s \leq 1} B(s) > cb/\sqrt{2\sigma t} \}$ and
  $v = \int_{t}^\infty \Pr\{X > u\}du - t\Pr\{X > t\}$) to obtain
  \begin{align*}
    \int_{b\delta_+}^\infty t \Pr \left\{ \max_{0 \leq s
    \leq 1} \sigma B(s) > \frac{cb}{\sqrt{2t}}  \right\} \Pr\{ X \in
    dt\} &= O\left( b \int_{b\delta_+}^\infty
           \frac{\Pr\{X>t\}}{\sqrt{t}} \exp \left(
           -\frac{cb^2}{4\sigma 
           t}\right) dt \right),
  \end{align*}
  which, in turn, is $O(b \times b\bar{B}(b^2))$ because of Lemma
  \ref{LEM-BSQ-SIMP-TERM}. Therefore, due to \eqref{INTER-B31-1} and
  \eqref{INTER-B31-2}, along with the observation that
  $\Pr\{ X > b \delta_+\} = \Theta(\bar{B}(b))$ (due to regular
  variation), we obtain
  \begin{align}
    \mathbb{E}\left[ I\left( \max_{0\leq n \leq N_A(X) + 1}2\left\vert S_{n}\right\vert
    >\left( \delta -\delta _{-}\right) b, N_A(X) + 1 \leq 2X\right) X \
    \left|\frac{}{}\right.  X > b\delta_+ \right] = O \left( b^2
    \bar{B}(b) + b^2 \frac{\bar{B}(b^2)}{\bar{B}(b)}\right). 
    \label{INTER-B31-3}
  \end{align}
  On the other hand, given that $N_A(t)/t \rightarrow 1$ as
  $t \rightarrow \infty,$ the event $\{N_A(t) > 2t-1\}$ corresponds
  to a large deviations event with exponentially small probability for
  large values of $t.$ Therefore, we have that
   \begin{align*}
     &\mathbb{E}\left[ I\left( \max_{0\leq n\leq N_A(X)}2\left\vert S_{n}\right\vert
       >\left( \delta -\delta _{-}\right) b, N_A(X) + 1> 2X\right) X \
       \left|\frac{}{}\right.  X > b\delta_+ \right]\\
     &\quad\quad \quad\quad \leq
       \int_{b\delta_+}^\infty t \Pr \left\{ N_A \left( t \right) >
       2t - 1\right\} \Pr \left\{ X \in dt \right\} = O \left( \exp
       \left(-\gamma b \right)\right),
   \end{align*}
   for a suitable $\gamma > 0.$ This observation, along with
   \eqref{INTER-B31-3}, concludes the proof of Lemma
   \ref{Lem_No_Run_Max_Mult}.
\end{proof}

\noindent The proof of Lemma \ref{Lem_Run_Max_Appers}, where running
maximum appears twice, is similar, but more involved, and is presented
in Appendix \ref{APP-OTHERS}, so that we can continue with central
arguments in the main body of the paper. Before moving to Part 3.b) of
the proof, it is important to note that Lemma \ref{Lem_Simplifcation}
stands proved as an immediate consequence of Lemmas
\ref{Lem_No_Run_Max_Mult} and \ref{Lem_Run_Max_Appers}.

\subsection*{Part 3.b) Simplifications using uniform large deviations:
  the $\alpha \in (1,2)$ case}
We shall leverage much of the reasoning behind Part 3.a) and prove the
following result:
\begin{lemma}
\label{LEM-SIMP-INF-VAR}
Suppose that $\bar{B}\left( x\right) \sim cx^{-\alpha }$ as $x\rightarrow
\infty $ for some $c>0$ and $\alpha \in \left( 1,2\right) $. Also, suppose
that Assumption \ref{ASSUMP_LEFT_TAIL} holds. Then, 
\begin{equation*}
  B_{3}\left( b\right) =O\left( \frac{b^{\alpha }\bar{B}\left( b^{\alpha
        }\right)}{\bar{B}(b)} 
  \right).
\end{equation*}
\end{lemma}
\noindent We begin with a uniform convergence result which is a
special case of Theorem 3.8.2 of \cite{MR2424161}:
\begin{lemma}
  \label{LEM-MAX-SUM-ASYMP-INF-VAR}
  Suppose that $\bar{B}\left( x\right) \sim cx^{-\alpha }$ as
  $x\rightarrow \infty $ for some $c>0$ and
  $\alpha \in \left( 1,2\right) $. Also, suppose that Assumption
  \ref{ASSUMP_LEFT_TAIL} holds.Then, there exists a positive integer
  $m_0$ such that for all $m \geq m_0,$
\begin{equation*}
  \mathbb{P}\left\{ \max_{0\leq n\leq m}\left\vert S_{n}\right\vert >x\right\}
  \leq 3 \mathbb{P}\left\{ Z_*>\frac{x}{
      (cm)^{1/\alpha }}\right\},
\end{equation*}
where $Z_* := \max_{0 \leq s \leq 1}Z(s) $ is the maximum of an
$\alpha $-stable process $(Z(t): 0 \leq t \leq 1)$ satisfying
$\Pr\{ Z(1) > x\} \sim x^{-\alpha}$ as $x \rightarrow \infty.$
Additionally, for such a stable process $Z(\cdot),$ we have that
\[\Pr\{ Z_* > x\} \sim x^{-\alpha} \text{ as } x \rightarrow \infty.\]
\end{lemma}
\noindent The adaptation of Theorem 3.8.2 of \cite{MR2424161} to the
case where maximum of $\vert S_n \vert$ appears (instead of maximum of
$S_n$) is similar to the argument in the proof of Lemma
\ref{LEM-MAX-ABS-SUM-ASYMP} in Part 3.a), and therefore is
omitted. The dominant contribution to $B_3(b)$ is accounted for in
the following result:

\begin{lemma}
  \label{LEM-INF-VAR-TERM1}
  Suppose that $\bar{B}\left( x\right) \sim cx^{-\alpha }$ as
  $x\rightarrow \infty $ for some $c>0$ and
  $\alpha \in \left( 1,2\right) $. Also, suppose that Assumption
  \ref{ASSUMP_LEFT_TAIL} holds. Then,
\begin{equation*}
  \mathbb{E}\left[ I\left( \max_{0\leq k\leq N_A(X)+1}2\left\vert S_{n}\right\vert
      >\left( \delta -\delta _{-}\right) b\right) X \  \left \vert
        \frac{}{} \right. X>b\delta_+ \right]
    =O\left( \frac{b^{\alpha }\bar{B}\left( b^{\alpha
          }\right)}{\bar{B}(b)} \right). 
\end{equation*}
\end{lemma}

\begin{proof}[Proof of Lemma \ref{LEM-INF-VAR-TERM1}]
  As a consequence of Lemma \ref{LEM-MAX-SUM-ASYMP-INF-VAR}, we get
\begin{align}
  \mathbb{E}\left[ I\left( \max_{0\leq k\leq 2X}2\left\vert
  S_{n}\right\vert 
  >\left( \delta -\delta _{-}\right) b\right) X I\left( 
  X>b\delta_+\right) \right] &\leq 3\E\left[ \Pr\left\{ Z_* >
                               \frac{(\delta-\delta_-)b}{2\left(2cX\right)^{\frac{1}{\alpha}}}\right\} 
                               X I\left( X > b\delta_+\right)\right]
                               \nonumber\\ 
                             & = 3\E\left[ X I\left( X > \left(b\delta_+\right) \vee
                               \left(\frac{\bar{c} b}{Z_*}\right)^{\alpha}\right)\right] 
                               \label{INTER-INF-VAR-TERM1}
\end{align}
where $\bar{c} := (\delta-\delta_-)/(2^{\alpha+1}c)^{1/\alpha}.$
Additionally, for all large enough $x,$ there exists a constant $C$
such that $\E[XI(X > x)] \leq Cx^{-(\alpha-1)},$ because of Karamata's
theorem and the observation that $\Pr\{ X > x\} \sim cx^{-\alpha}$ as
$x \rightarrow \infty.$ Therefore, for all $b$ large enough, we obtain
\begin{align*}
  \E\left[ X I\left( X > \left(b\delta_+\right) \vee
  \left(\frac{\bar{c} b}{Z_*}\right)^{\alpha}\right)\right] &\leq C\E
                                                              \left[ \left(\left(b\delta_+\right) \vee
                                                              \left(\frac{\bar{c}
                                                              b}{Z_*}\right)^{\alpha}\right)^{-(\alpha-1)}\right]\\
                                                            &\leq 
                                                              C\left(b\delta_+\right)^{-(\alpha-1)}\Pr\left\{ Z_* >
                                                              \frac{\bar{c}b^{1-\frac{1}{\alpha}}}{\delta_+^{\frac{1}{\alpha}}}
                                                              \right\}
                                                              + C
                                                              \left(
                                                              \bar{c}
                                                              b\right)^{-\alpha(\alpha-1)}
                                                              \E\left[
                                                              Z_*^{\alpha^2-\alpha}\right] ,
\end{align*}
which, in turn, is $O(b^\alpha\bar{B}(b^\alpha))$ because
$\E[Z_*^{\alpha^2-\alpha}] < \infty$ when $\alpha \in (1,2).$
Therefore, due to \eqref{INTER-INF-VAR-TERM1},
\begin{align*}
  \mathbb{E}\left[ I\left( \max_{0\leq k\leq 2X}2\left\vert
  S_{n}\right\vert 
  >\left( \delta -\delta _{-}\right) b, \ N_A(X) + 1 \leq 2X \right) X I\left( 
  X>b\delta_+\right) \right] = O\left( b^{\alpha}\bar{B}\left(
  b^\alpha \right)\right).  
\end{align*}
On the other hand, the event $\{ N_A(b) + 1 > 2b\}$ is a large
deviations event with probabilities exponentially decaying in $b,$ and
as argued in the proof of Lemma \ref{Lem_No_Run_Max_Mult},
\begin{align*}
  \mathbb{E}\left[ I\left( \max_{0\leq k\leq 2X}2\left\vert
  S_{n}\right\vert 
  >\left( \delta -\delta _{-}\right) b, \ N_A(X) + 1 > 2X \right) X I\left( 
  X>b\delta_+\right) \right] = O\left( \exp(-\gamma b)\right),
\end{align*}
for a suitable $\gamma > 0.$ These two observations, after adjusting
for the conditioning by dividing by
$\Pr\{ X > b\delta_+\} = \Theta(\bar{B}(b)),$ prove Lemma
\ref{LEM-INF-VAR-TERM1}.
\end{proof}

\begin{lemma}
  \label{LEM-INF-VAR-TERM2}
  Suppose that $\bar{B}\left( x\right) \sim cx^{-\alpha }$ as
  $x\rightarrow \infty $ for some $c>0$ and
  $\alpha \in \left( 1,2\right) $. Also, suppose that Assumption
  \ref{ASSUMP_LEFT_TAIL} holds. Then,
\begin{equation*}
  \mathbb{E}\left[ I\left( \max_{0\leq k\leq N_A(X)+1}2\left\vert S_{n}\right\vert
      >\left( \delta -\delta _{-}\right) b\right)  \max_{0\leq k\leq
      N_A(X)+1} \left\vert S_{n}\right\vert \  \left \vert
        \frac{}{} \right. X>b\delta_+ \right]
    =O\left( b^2\bar{B}\left( b \right) \right). 
\end{equation*}
\end{lemma}

\noindent As in Part 3.a), the proof of Lemma \ref{LEM-INF-VAR-TERM2}
is furnished in Appendix \ref{APP-OTHERS}. The main result of this
section, Lemma \ref{LEM-SIMP-INF-VAR}, which aims to prove that
$B_3(b) = O(b^\alpha\bar{B}(b^\alpha)/\bar{B}(b))$ is an immediate
consequence of Lemmas \ref{LEM-INF-VAR-TERM1} and
\ref{LEM-INF-VAR-TERM2}, along with the observation that
$b^2\bar{B}^2(b) = o(b^\alpha\bar{B}(b^\alpha))$ when $\alpha < 2.$

\subsection*{Part 4) Estimation of $B_2(b)$} 
The objective of this subsection is to prove Lemma \ref{Lem_est_B1},
and subsequently, complete the proof of Proposition \ref{PROP-UB}.
\begin{lemma}
\label{Lem_est_B1}
Suppose that Assumptions \ref{ASSUMP-DIST-V} and
\ref{ASSUMP_LEFT_TAIL} hold, and that $\rho =1$. Then,
\begin{equation*}
B_{2}\left( b\right) =O\left( b^{2}\bar{B}\left( b\right) ^{2}\right).
\end{equation*}
\end{lemma}

\noindent It follows from the definition of $B_2(b)$ that
\begin{align*}
  B_2(b) &= \E_{\bf 0}\left[ \sum_{k=0}^{\tau_0-1} I\left( W_k^{(1)} >
           b\right) I \left( \bar{\tau}_{b\delta_+}^{(2)} \leq
           \tau_{b\delta}^{(1)}, \tau_{b\delta_{-}}^{(2)} <
           \tau_0\right)\right]\\
         &= \E_{\bf 0}\left[ I \left( \bar{\tau}_{b\delta_+}^{(2)} 
           \leq \tau_{b\delta}^{(1)}, \tau_{b\delta_{-}}^{(2)} <
           \tau_0\right)   \E_{\bf 0} \left[ \sum_{k=\bar{\tau}_{b\delta_+}^{(2)}}^{\tau_0-1}I\left(
           W_k^{(1)} >  b\right) \ \left \vert \frac{}{}
           \right. \mathcal{F}_{\bar{\tau}_{b\delta_+}^{(2)}}
           \right]\right]
           \end{align*}
Then due to the Markov property of ${\bf W},$ we get
\begin{align}
  B_2(b)     &= \E_{\bf 0}\left[ I \left( \bar{\tau}_{b\delta_+}^{(2)}
               \leq \tau_{b\delta}^{(1)}, \tau_{b\delta_{-}}^{(2)} <
               \tau_0\right)   \E_{{\bf W}_{\bar{\tau}_{b\delta_+}^{(2)}}}
               \left[ \sum_{k=0}^{\tau_0-1}I\left( 
               W_k^{(1)} >  b\right)   \right]\right].
               \label{B2-INTER}
\end{align}
\subsection*{Evaluation of inner expectation} Due to a similar
conditioning with respect to $\mathcal{F}_{\tau_b^{(2)}},$ we obtain
\begin{align*}
  \E_{{\bf W}_{\bar{\tau}_{b\delta_+}^{(2)}}}
  \left[ \sum_{k=0}^{\tau_0-1}I\left( 
  W_k^{(1)} >  b\right)   \right] &= \E_{{\bf W}_{\bar{\tau}_{b\delta_+}^{(2)}}}
                                    \left[ I\left( \tau_b^{(2)} <
                                    \tau_0\right)
                                    \E_{{\bf W}_{\tau_b^{(2)}}}\left[\sum_{k=\tau_b^{(2)}}^{\tau_0-1}I\left(  
                                    W_k^{(1)} >  b\right)\right]
                                    \right] \\
                                  &\leq \E_{{\bf W}_{\bar{\tau}_{b\delta_+}^{(2)}}}
                                    \left[ I\left( \tau_b^{(2)} <
                                    \tau_0\right) \E_{{\bf
                                    W}_{\tau_b^{(2)}}}\left[ \tau_0
                                    \right] \right],
\end{align*}
which, due to Lemma \ref{Lem_LB_ET0}, admits the following upper
bound: 
\begin{align}
  &\E_{{\bf W}_{\bar{\tau}_{b\delta_+}^{(2)}}}
    \left[ \sum_{k=0}^{\tau_0-1}I\left( 
    W_k^{(1)} >  b\right)   \right] \leq \E_{{\bf W}_{\bar{\tau}_{b\delta_+}^{(2)}}}
    \left[ I\left( \tau_b^{(2)} <
    \tau_0\right) \left(C_1
    W_{\tau_b^{(2)}}^{(2)} + C_0 \right)
    \right] \nonumber\\
  &\quad\quad= C_1 \E_{{\bf W}_{\bar{\tau}_{b\delta_+}^{(2)}}}
    \left[ W_{\tau_b^{(2)}}^{(2)} \
    \left \vert \frac{}{} \right. 
    \tau_{b}^{(2)} <
    \tau_0 \right] \Pr_{{\bf W}_{\bar{\tau}_{b\delta_+}^{(2)}}} \left\{
    \tau_b^{(2)} < \tau_0 \right\} +
    C_0 \Pr_{{\bf
    W}_{\bar{\tau}_{b\delta_+}^{(2)}}}
    \left\{ \tau_b^{(2)} < \tau_0 \right\} \nonumber\\
  &\quad\quad \leq C_2 b  \Pr_{{\bf W}_{\bar{\tau}_{b\delta_+}^{(2)}}} \left\{
    \tau_b^{(2)} < \tau_0 \right\} 
    \label{B2-INTER-2}
\end{align}
for some positive constant $C_2,$ because, due to Lemma
\ref{LEM_ST_DOM},
\begin{align*}
  \E_{{\bf W}_{\bar{\tau}_{b\delta_+}^{(2)}}} \left[
  W_{\tau_b^{(2)}}^{(2)} \ \left \vert \frac{}{}
  \right. \tau_{b}^{(2)} < \tau_0 \right] \leq \E \left[ X + b\
  \left\vert\frac{}{}\right. X > b \right] = O(b).
\end{align*}

\noindent Next, we obtain a bound for
$\Pr_{\bf w} \{ \tau_b^{(2)} < \tau_0\},$ and use it in
\eqref{B2-INTER-2}. 

\begin{lemma}
\label{LEM_LB_PROB}
Suppose that Assumptions \ref{ASSUMP-DIST-V} and
\ref{ASSUMP_LEFT_TAIL} hold, and that $\rho =1$. Let
$\delta _{+}<1/2$, then there exists a constant $C>0$ such that for
all ${\bf w}=\left( w_{1},w_{2}\right) $ satisfying
$w_1 \leq w_{2}<b\delta _{+}$ we have
\begin{equation*}
  \Pr_{\bf w}\left\{ \tau _{b}^{\left( 2\right) }<\tau _{0}\right\} \leq
  C\left( w_{2}+1\right) \bar{B}\left( b\right) .
\end{equation*}
\end{lemma}
\noindent The proof of Lemma \ref{LEM_LB_PROB} is instructive, as it
employs Lyapunov bound techniques to derive the above uniform
bound. The arguments involved are different from the rest of the
paper, and the whole of Appendix \ref{APP-LYAPUNOV} is dedicated to
expose the techniques clearly. Now, we aim to complete the proof of
Lemma \ref{Lem_est_B1}. Due to Lemma \ref{LEM_LB_PROB} and
\eqref{B2-INTER-2},
\begin{align*}
  \E_{{\bf W}_{\bar{\tau}_{b\delta_+}^{(2)}}}
  \left[ \sum_{k=0}^{\tau_0-1}I\left( 
  W_k^{(1)} >  b\right)   \right] \leq C C_2b
  \left( W^{(2)}_{\bar{\tau}_{b\delta_+}^{(2)}} + 1\right) \bar{B}(b)
  \leq CC_2b\left(b\delta_+ + 1\right) \bar{B}(b),
\end{align*}
and therefore, due to \eqref{B2-INTER} and a similar application of
Lemma 
\ref{LEM_LB_PROB} with ${\bf w} = {\bf 0},$
\begin{align*}
  B_2(b) \leq C C_2 \Pr_{\bf 0} \left\{ \tau_{b\delta_{-}}^{(2)} <
  \tau_0\right\}b\left(b\delta_+ + 1\right) \bar{B}(b)  = O\left( b^2
  \bar{B}^2(b) \right).
\end{align*}
This concludes the proof of Lemma \ref{Lem_est_B1}. 

\begin{proof}[Proof of Proposition \ref{PROP-UB}]
  We apply Lemma \ref{LEM_LB_PROB} with ${\bf w} = {\bf 0}$ to the
  bound for $B_1(b)$ in Lemma \ref{LEM_RED_RW} as well. This, due to
  Lemmas \ref{Lem_Simplifcation} and \ref{LEM-SIMP-INF-VAR}, results
  in
  \begin{align*}
    B_1(b) = O\left( \bar{B}\left( b\delta_+\right) \times B_3(b)
    \right) = \begin{cases}
      O\left( b^2\bar{B}^2(b) + b^2\bar{B}\left( b^2 \right) \right)
      \quad \text{if } \alpha > 2, \\ 
      O\left( b^\alpha\bar{B}\left( b^\alpha \right) \right)
      \quad\quad\quad\quad\quad 
      \text{ if } \alpha \in (1,2). 
      \end{cases}
  \end{align*}
  Additionally, we have that $B_2(b) = O(b^2\bar{B}(b)^2)$ and
  $\E_{\bf 0} [\tau_0] < \infty,$ respectively, from Lemmas \ref{Lem_est_B1}
  and \ref{Lem_LB_ET0}. Therefore, from \eqref{RR-INTERMSOF_B}, we
  arrive at the statement of Proposition \ref{PROP-UB}, and this
  concludes the proof.
\end{proof}

\appendix
\section{Lyapunov bound techniques for a uniform bound on
  $\Pr_{\bf w} \{ \tau_b^{(2)} < \tau_0\}$}
\label{APP-LYAPUNOV}
\noindent We use the Lyapunov bound technique that has been employed
in \cite{blanchet2008}, \cite{FLUID_HEUR_QUESTA}, and
\cite{Denisov20133027}. 
The strategy is to define a Markov kernel
$Q_{\theta }\left( {\bf w},\cdot \right) $ (indexed by some parameter
$\theta $) and a non-negative function
$H_{b}\left( w_{1},w_{2}\right) $ satisfying the following conditions:
\begin{itemize}
\item[(L1)] For every ${\bf w} = (w_1,w_2)$ such that $w_2 < b,$
  \begin{equation*}
    \mathbb{E}_{\bf w}^{\theta }\left[ r_\theta\left( {\bf w} ,{\bf
          W}_{1}\right) H_{b}\left( {\bf W}_{1}\right) \right] \leq H_{b}\left( 
      {\bf w} \right),  \label{LYA_IN_PART1}
  \end{equation*}
  where $\Pr_{\bf w}\{ {\bf W}_1 \in \cdot\}$ is the nominal
  transition kernel induced by recursions \eqref{MIN-REC} and
  \eqref{MAX-REC},
  $r_\theta ({\bf w}, {\bf x}) := \Pr_{\bf w} \{ {\bf W}_1 \in d{\bf
    x}\}/Q_\theta ( {\bf w}, d{\bf x})$
  is the corresponding Radon-Nikodym derivative with respect to
  $Q_\theta({\bf w}, \cdot)$, and
  $\mathbb{E}_{{\bf w}}^{\theta }\left[ \cdot \right]$ is the
  expectation associated with the probability measure in path space
  for the Markov evolution induced by
  $Q_{\theta }\left( {\bf w},\cdot \right).$
 
\item[(L2)] Whenever ${\bf w} = (w_1,w_2)$ is such that $w_2 > b,$
  $H_{b}\left( w_{1},w_{2}\right) \geq 1.$
\end{itemize}
If conditions (L1) and (L2) are satisfied, then following
the analysis in Part (iii) of Theorem 2 of \cite{blanchet2008}, we
have that 
\begin{equation}
  \mathbb{P}_{\bf w} \left\{\tau _{b}^{\left( 2\right) }<\tau _{0}
  \right\} \leq \mathbb{E}
  _{\bf w}^{\theta } \left[ \prod\limits_{n=0}^{\tau _{b}^{\left(
          2\right) }-1} r_{\theta }\left(
      {\bf W}_{n},{\bf W}_{n+1}\right) 
    H_{b}\left( {\bf W}_{\tau_b^{(2)}}\right) I\left(\tau _{b}^{\left( 2\right) }<\tau
      _{0} \right)\right] \leq
  H_{b}\left( w_{1},w_{2}\right) .  \label{IN_LB_END}
\end{equation}
\noindent The construction of
$Q_{\theta }\left( {\bf w},\cdot \right) $ and
$H_{b}\left( \cdot \right) $ follows the intuition explained in
\cite{blanchet2008} and \cite{FLUID_HEUR_QUESTA}: We wish to select
$Q_{\theta }\left( {\bf w},\cdot \right) $ as closely as possible to
the conditional distribution of the process
$\left\{ {\bf W}_{n}:n\geq 0\right\} $ given that
$\{\tau _{b}^{\left( 2\right) }<\tau _{0}\},$ because in that case, it
happens that (\ref{IN_LB_END}) is automatically satisfied with
equality. Additionally, we shall find a suitable non-negative function
$G_{b}\left( \cdot \right) $ so that
$H_{b}\left( w_{1},w_{2}\right) =G_{b}\left( w_{1}+w_{2}\right) $
satisfies the Lyapunov inequality (L1).

\noindent For ease of notation, let us write
\[l :=w_{1}+w_{2},\ L := W_1^{(1)} + W_2^{(2)} \text{ and } \Delta :=
L - l.\]
In order to construct $Q_{\theta }\left( {\bf w},\cdot \right) $ and
$G_{b}\left( \cdot \right),$ first define the Markov transition kernel
\begin{align*}
  Q^{\prime }\left( {\bf w},A\right) &=\mathbb{P}_{\bf w}\left\{ {\bf W}_{1}\in A\text{ }
                                       |\ X_1 >a\left( b-l\right) \right\}
                                       p\left( {\bf w}\right)\\ 
                                     &\quad\quad\quad\quad  +\mathbb{P}_{\bf w}\left\{
                                       {\bf W}_{1}\in A\text{ }|\ X_1 \leq
                                       a\left( b-l\right) ,W^{(2)}_1
                                       >0\right\}  \left(1-p\left(
                                       {\bf w}\right) \right),  
\end{align*}
where $p\left( {\bf w}\right) $ will be specified momentarily, and the
choice $a \in (0,1)$ is arbitrary. On the set
$\{\tau _{b}^{\left( 2\right) }<\tau _{0}\}$, given
${\bf w}=\left( w_{1},w_{2}\right) $ with $w_1 \leq w_{2} < b$, we
have that the nominal kernel $\Pr_{\bf w}\{ {\bf W}_1 \in \cdot\}$ is
absolutely continuous with respect to
$Q^{\prime }\left( {\bf w},\cdot \right).$ Now, for $z \geq 0,$ define
\begin{align*}
  h_{b}\left( z\right) = \int_{0}^{z+\kappa _{0}}\mathbb{P}\left\{
  X>b-z+t\right\}dt = \int_{b-z}^{b+\kappa _{0}}\mathbb{P}\left\{ X>u\right\} du.
\end{align*}
Next, write
\begin{equation*}
G_{b}\left( l\right) =\min (\kappa _{1}h_{b}\left( l\right) ,1)
\end{equation*}
and set 
\begin{equation*}
  p\left( {\bf w} \right) = \frac{\mathbb{P}\left\{ X>a\left( b-l\right)
    \right\} }{\kappa_2 h_{b}\left( l\right) }
\end{equation*}
where $\kappa_2$ is a number larger than
\begin{align*}
  \sup_{x > 0} \frac{\Pr \left\{ X > a
  x\right\}}{\int_{x}^{x+l+\kappa_0} \Pr\left\{ X > u\right\}du} <
  \infty. 
\end{align*}
Finally, define
$\theta =\left( \kappa_{0},\kappa_{1},\kappa_{2}\right) $ and write
\begin{equation*}
  Q_{\theta }\left( {\bf w},\cdot \right) =Q^{\prime }\left( {\bf w},\cdot \right)
  I(G_{b}\left( l\right) <1)+K\left( {\bf w},\cdot \right) I(G_{b}\left( l\right)
  =1).
\end{equation*}
Recall the notation
$l = w_1 + w_2 \text{ and } L =W_{1}^{\left( 1\right) }+W_{1}^{\left(
    2\right) }.$
Condition (L1) is verified via the following proposition:
\begin{proposition}
  \label{PROP-L1-VERIF}
  For every ${\bf w} = (w_1,w_2)$ such that $w_2 < b,$ we have that
  \begin{align*}
    \E_{\bf w}^\theta\left[ r_\theta\left( {\bf w}, {\bf W}_1\right)
    G_b(L) \right] \leq G_b(l).
  \end{align*}
\end{proposition}

\noindent For proving Proposition \ref{PROP-L1-VERIF}, we consider
only the case $G_b(l) < 1.$ When $G_b(l) = 1,$ the inequality is
satisfied trivially. The following results are crucial in the proof of
Proposition \ref{PROP-L1-VERIF}.

\begin{lemma}
  \label{LEM-LYAP-MEAN-NEG}
  There exist positive constants $\mu$ and $C$ such that
  \[ \E_{(w_1,w_2)} \left[ \Delta I\left( W_1^{(2)} > 0 \right) \right] <
  -\mu\] whenever $w_2 > C.$
\end{lemma}

\begin{proof}
  First, observe that
  \begin{align*}
    \E_{\bf w} \left[ \Delta I\left( W_1^{(2)} > 0 \right) \right] =
    \E_{\bf w} \left[ \Delta \right] -  \E_{\bf w} \left[ \Delta
    I\left( W_1^{(2)} = 0 \right) \right].
  \end{align*}
  Additionally, note that $\Delta = -(w_1+w_2)$ when $W_1
^{(2)} = 0.$ Therefore, 
  \begin{align*}
    \E_{(w_1,w_2)} \left[ \Delta I\left( W_1^{(2)} = 0 \right) \right]
    = -\left( w_1+w_2 \right) \Pr \left\{ w_1 + V - T \leq 0, \   w_2 - T
    \leq 0 \right\}. 
  \end{align*}
  Therefore, due to Lemma \ref{LEM-LYAP-ET0},
  \begin{align*}
    \E_{\bf w} \left[ \Delta I\left( W_1^{(2)} > 0 \right) \right] 
    &\leq  \E_{\bf w} \left[ \Delta \right] + \left( w_1+w_2 \right) \Pr \left\{ w_2 - T
      \leq 0 \right\}\\
    & \leq -\epsilon + 2w_2 \Pr \left\{ T > w_2 \right\},
  \end{align*}
  where $w_2\Pr \{ T > w_2\}$ can be made arbitrarily small by
  choosing $C > w_2$ large enough. Hence the claim stands verified.
\end{proof}

%

\begin{lemma}
  \label{LEM-LYAP-DOM-CONV}
  Recall that $l = w_1+w_2.$ The following holds as $(b-l) \rightarrow
  \infty:$
    \begin{align*}
      &\mathbb{E}_{\bf w}^{\theta }\left[ r_{\theta}\left( {\bf w},{\bf
        W}_{1}\right)\frac{G_{b}\left( L \right)}{G_b(l)} I\left(X_1
        \leq   a\left( b-l\right) \right) \right]  \nonumber\\
      &\quad\quad
        \leq \mathbb{P}_{\bf w}\left\{ W_{1}^{\left(
        2\right) }>0\right\} + \frac{\Pr\left\{ X > b-l
        \right\}}{h_b(l)}\mathbb{E}_{\bf w}\left[ \Delta I\left(
        W_{1}^{\left( 2\right)}>0\right)\right] \left(1 +
        o(1)\right). 
  \end{align*}
\end{lemma}

\begin{proof}
  Since
  \begin{equation*}
    G_{b}\left( L_{1}\right) =G_{b}\left( l\right) +\int_{0}^{1}G_{b}^{\prime
    }\left( l+u\Delta \left( 1\right) \right) \Delta \left( 1\right) du,
  \end{equation*}
  we introduce a uniform random variable $U$, independent of
  everything else, to write
  \begin{align}
    &\mathbb{E}_{\bf w}^{\theta }\left[ r_{\theta}\left( {\bf w},{\bf
      W}_{1}\right)\frac{G_{b}\left( L \right)}{G_b(l)} I\left(X_1
      \leq   a\left( b-l\right) \right) \right]  \nonumber\\
    &\quad\quad =\mathbb{E}_{\bf w}\left[ \frac{G_{b}\left( L \right)}{G_b(l)}
      I\left(X_1 \leq a\left( b-l\right) ,W_{1}^{\left( 2\right) }>0
      \right)\right] \nonumber\\
    &\quad\quad=\mathbb{E}_{\bf w}\left[ \left( 1
      +\frac{G_{b}^{\prime }\left( l+U\Delta \right)}{G_b(l)} \Delta
      \right) I\left( X_1 \leq a\left( 
      b-l\right) ,W_{1}^{\left( 2\right) }>0\right)\right] \nonumber\\
    &\quad\quad
      \leq \mathbb{P}_{\bf w}\left\{ W_{1}^{\left(
      2\right) }>0\right\} + \frac{\Pr\left\{ X > b-l
      \right\}}{h_b(l)}\mathbb{E}_{\bf w}\left[ 
      \frac{\Pr \left\{ X >  b-l-U\Delta \right\}
      }{\Pr\left\{ X > b-l \right\}}\Delta I\left(X_1\leq a\left( b-l\right)
      ,W_{1}^{\left( 2\right)
      }>0\right)\right]. 
   \label{LYAP-J2-INTER-1}
  \end{align}
  We have used
  $G_{b}^{\prime }\left( l+U\Delta \left( 1\right) \right) =\kappa_1
  \Pr \left\{ X > b-l-U\Delta \left( 1\right) \right\}$
  to write the last step.  Additionally, whenever
  $X_1\leq a\left( b-l\right),$ observe that 
  \begin{align*}
    \Delta =\left( w_{1}+X_1 \right) ^{+}-w_{1}+\left(
    w_{2}-T_1 \right) ^{+}-w_{2}\leq X_1^{+}\leq a\left(
    b-l\right),
  \end{align*}
  and therefore,
  $\Pr \left\{ X > b-l-U\Delta \right\} \leq \Pr \left\{ X >
    (1-a)(b-l) \right\} \leq m_{1-a} \Pr \left\{ X > b - l
  \right\},$ where
  \[ m_t := \sup_{x > 0}\frac{\Pr\{ X > t x\}}{\Pr\{ X > x\}} <
  \infty,\]
  for every $t > 0.$ Here, the finiteness of $m_t$ follows from the
  regularly varying nature of the tail distribution of $X$ (recall
  that $\Pr\{ X > x\} \sim \bar{B}(x)$ as $x \rightarrow \infty$).  As
  a result, we have the following uniform bound for various values of
  $b$ and $l:$
 \begin{align}
   \mathbb{E}_{\bf w}\left[ 
   \frac{\Pr \left\{ X >  b-l-U\Delta \right\}
   }{\Pr\left\{ X > b-l \right\}}\Delta I\left(X_1\leq a\left(
   b-l\right),W_{1}^{\left( 2\right) }>0\right)\right] \leq m_{1-a} \E
   X^+. 
    \label{LYAP-J2-INTER-BND-2}
 \end{align}
 Consequently, due to dominated convergence theorem, we obtain that 
 \begin{align*}
   \mathbb{E}_{\bf w}\left[ 
   \frac{\Pr \left\{ X >  b-l-U\Delta \right\}
   }{\Pr\left\{ X > b-l \right\}}\Delta I\left(X_1\leq a\left( b-l\right)
   ,W_{1}^{\left( 2\right)
   }>0\right)\right] \sim \E_{\bf w} \left[ \Delta I \left(
   W_1^{(2)} > 0 \right)
   \right],
 \end{align*}
 as $(b-l) \rightarrow \infty.$ Now, the statement of Lemma
 \ref{LEM-LYAP-DOM-CONV} is immediate from \eqref{LYAP-J2-INTER-1} and
 the above stated convergence.
\end{proof}

\begin{proof}[Proof of Proposition \ref{PROP-L1-VERIF}]
  As mentioned before, we consider $G_{b}\left( l\right) <1.$ First,
  observe that
  \begin{align}
    \mathbb{E}_{\bf w}^{\theta }\left[ r_{\theta} \left( {\bf w},
    {\bf W}_{1}\right) \frac{G_{b}\left( L \right) }{G_b(l)}I\left(X_1>a\left( b-l\right)\right)
    \right] &= \mathbb{E}_{\bf w}\left[ \frac{G_{b}\left( L \right)
              }{G_b(l)}I\left(X_1 >a\left( b-l\right)\right)
              \right] \nonumber \\
            &\leq \frac{\mathbb{P}\left\{ X>a\left( b-l\right)
              \right\}}{\kappa_{1} h_b(l)} 
    \label{LYAP-INTER-J1}
  \end{align}
  because $G_b(\cdot) \leq 1.$ For a respective bound on the
  complementary event $\{ X_1 \leq a(b-l) \},$ it is easy to see that
  our strategy must use Lemmas \ref{LEM-LYAP-MEAN-NEG} and
  \ref{LEM-LYAP-DOM-CONV} in the following way: Given $\delta > 0,$
  there exists a constant $C_\delta$ large enough such that for all
  initial conditions ${\bf w} = (w_1,w_2)$ satisfying $w_2 > C$ and
  $b-l > C_\delta,$
  \begin{align*}
    \mathbb{E}_{\bf w}^{\theta }\left[ r_{\theta}\left( {\bf w},{\bf
      W}_{1}\right)\frac{G_{b}\left( L \right)}{G_b(l)} I\left(X_1
      \leq   a\left( b-l\right) \right) \right]  
          \leq \mathbb{P}_{\bf w}\left\{ W_{1}^{\left(
      2\right) }>0\right\} - (1-\delta) \mu \frac{\Pr\left\{ X > b-l
      \right\}}{h_b(l)}.
  \end{align*}
  Combining this bound with \eqref{LYAP-INTER-J1}, we obtain
  \begin{align*}
    \mathbb{E}_{\bf w}^{\theta }\left[ r_{\theta}\left( {\bf w},{\bf
    W}_{1}\right)\frac{G_{b}\left( L \right)}{G_b(l)} \right]  
    &\leq \mathbb{P}_{\bf w}\left\{ W_{1}^{\left( 2\right) }>0\right\}
      + \frac{\Pr\left\{ X > a(b-l) \right\}}{h_b(l)} \left( \frac{1}{\kappa_1}
      - \frac{(1-\delta) \mu}{m_a}\right)\\
    &\leq 1 + \kappa_2 p\left( {\bf w} \right) \left( \frac{1}{\kappa_1}
      - \frac{(1-\delta) \mu}{m_a}\right)
  \end{align*}
  which is, in turn, smaller than $1$ for $\kappa_1$ suitably large.
  In addition to this, in the region
  $\{(w_1,w_2): w_2 > C, b - l < C_\delta\},$ we simply let
  $G_b(l) = 1$ by again choosing $\kappa_1$ large enough. This
  flexibility in the choice of $\kappa_1$ yields us
  \begin{align}
    \label{LYAP-J2-IMP-REG}
    \mathbb{E}_{\bf w}^{\theta }\left[ r_{\theta}\left( {\bf w},{\bf
    W}_{1}\right)\frac{G_{b}\left( L \right)}{G_b(l)} \right] \leq 1
  \end{align}
  for initial conditions ${\bf w} = (w_1,w_2)$ satisfying $w_2 > C.$
  Now, turning our attention to the values of ${\bf w}$ such that
  $w_2 \leq C,$ we see that $l = w_1+ w_2 \leq 2C,$ and as a consequence
  of the regularly varying nature of the tail of $X,$ we obtain
  \begin{align*}
    \frac{\Pr \left\{ X > b-l\right\}}{h_b(l)} =
    \left( \int_{0}^{l+\kappa_0} \frac{\Pr\left\{ X > b-l +
    u\right\}}{\Pr \left\{ X > b-l\right\}} du\right)^{-1} =
    \frac{1+o(1)}{l+\kappa_0}  
  \end{align*}
  as $b \rightarrow \infty.$ Then, it is immediate from
  \eqref{LYAP-J2-INTER-1} that whenever $w_2 \leq C,$
  \begin{align*}
    \mathbb{E}_{\bf w}^{\theta }\left[ r_{\theta}\left( {\bf w},{\bf
    W}_{1}\right)\frac{G_{b}\left( L \right)}{G_b(l)} I\left(X_1  \leq
    a\left( b-l\right) \right) \right] 
    \leq \Pr_{\left(C,C\right)} \left\{ W_1^{(2)} > 0 \right\}
    + \frac{m_{1-a} \E \left[
    X_1^+\right]}{\kappa_0} \left(1+o(1) \right).
  \end{align*}
  Combining this bound with the one obtained in \eqref{LYAP-INTER-J1},
  we get
  \begin{align*}
    \mathbb{E}_{\bf w}^{\theta }\left[ r_{\theta}\left( {\bf w},{\bf
    W}_{1}\right)\frac{G_{b}\left( L \right)}{G_b(l)} \right]
    \leq \frac{p\left({\bf
    w}\right)\kappa_2}{\kappa_1} + \Pr_{\left(C,C\right)} \left\{
    W_1^{(2)} > 0 \right\}   + \frac{m_{1-a} \E \left[
    X_1^+\right]}{\kappa_0} \left(1+o(1) \right),
  \end{align*}
  which can also be made smaller than 1 by picking $\kappa_0$ and
  $\kappa_1$ large enough. Thus, for all initial conditions ${\bf w},$
  we have a consistent choice of parameters
  $(\kappa_0,\kappa_1,\kappa_2)$ that satisfies (L1).
\end{proof}
\noindent Since $G_b(l) = 1$ whenever $w_1+w_2 \geq b-C_\delta,$ we
also have $G_b(l) = 1$ if $w_2 > b.$ This verifies condition
(L2). Since both (L1) and (L2) are satisfied, it follows from
\eqref{IN_LB_END} that if $w_{2}<b\delta _{+}$ for some
$ \delta _{+}<1/2,$ then
\begin{align*}
  \Pr_{\bf w}\left\{ \tau _{b}^{\left( 2\right) }<\tau _{0}\right\}
  &\leq \kappa
    _{1}h_{b}\left( l\right) =\kappa _{1}\int_{b-l}^{b+\kappa _{0}}\mathbb{P}
    \left\{ X>u\right\} du \\
  &\leq\kappa _{1}\Pr \left\{ X > b-l \right\} \left( \kappa _{0}+l\right) \leq
    \kappa _{1}\bar{B}\left( b\left( 1-2\delta _{+}\right) \right) \left( \kappa
    _{0}+2w_{2}\right)\left(1+o(1)\right) .
\end{align*}
The right hand side of the previous inequality is equivalent to the statement
of Lemma \ref{LEM_LB_PROB}, so we conclude the proof.

\section{Proofs for other estimates}
\label{APP-OTHERS}
\begin{proof}[Proof of Lemma \ref{LEM-LYAP-ET0}]
  First, observe that 
  \begin{align*}
    \E \left[ (w_1 + V - T)^+ - w_1\right]  &= \E\left[ V - T \right]
                                              - \E \left[ (V-T)I\left(
                                              w_1 + V - T < 0 
                                              \right)\right] - w_1
                                              \Pr \left\{ w_1 + V -
                                              T < 0\right\}\\
                                            &= - \E \left[ (V-T)I\left(
                                              w_1 + V - T < 0 
                                              \right)\right] - w_1
                                              \Pr \left\{ w_1 + V -
                                              T < 0\right\}, \text{
                                              and }\\
    \E \left[ (w_2- T)^+ - w_2\right] &= -\E T + \E
                                        \left[ T I\left( w_2 -T <  0\right)\right] - w_2 \Pr\left\{ w_2
                                        -T < 0 \right\}.
  \end{align*}
  Then, it follows from the definition of ${\bf W}_{1}$ in recursions
  \eqref{MIN-REC} and \eqref{MAX-REC} that
  \begin{align*}
    &\E_{(w_1,w_2)}\left[ \left( W_1^{(1)} + W_1^{(2)}\right) - (w_1 +
      w_2) \right] = \E \left[ (w_1 + V - T)^+ - w_1\right] + \E\left[
      (w_2 + T)^+ -   w_2 \right]\\
    &\quad\quad = - \E \left[ V I\left( w_1 + V - T < 0
      \right)\right] - w_1  \Pr \left\{ w_1 + V -  T < 0\right\} -
      \E\left[ T I\left( w_1 + V - T \geq 0 \right)\right]\\
    &\quad\quad\quad\quad - w_2 \Pr\{
      w_2 - T < 0\} + \E \left[ T I\left( w_2 - T < 0 \right)\right] 
    \end{align*}
    which is negative if $\E[TI(T > w_2)]$ is small enough, and this
    can be achieved by choosing $C < w_2$ large enough. This completes
    the proof.
\end{proof}

\begin{proof}[Proof of Lemma \ref{LEM_SINGLE_Q_DOM}]
  From recursions \eqref{MIN-REC} and \eqref{MAX-REC}, it is evident
  that for every $1 \leq k \leq n,$
  \begin{align*}
    W_k^{(i)} \leq \left( W_{k-1}^{(i)} + X_{k}\right)^+, \quad i=1,2.
  \end{align*}
  We repeatedly expand the recursion, as below, to obtain
  \begin{align*}
    W_k^{(i)} &\leq \max\left\{ 0, \ W_{k-1}^{(i)} + X_k \right\}\\
              &\leq \max \left\{ 0, \ X_k, \ W_{k-2}^{(i)} + X_{k-1} +
                X_{k} \right\}\\
              &\leq \max \left\{ 0, \ X_k, \ X_{k-1} + X_k, \ 
 X_{k-2} +
                X_{k-1} + X_k, \ldots, \ W_{0}^{(i)} + X_1+ \ldots+
                X_{k-1} +   X_{k} \right\}\\ 
              &\leq S_k - \min_{0 \leq j \leq k}S_j + w_i,
  \end{align*}
  where we have used that $S_0:=0, S_j := X_1+\ldots+X_j$ and
  $w_i \geq 0.$ Then
  \begin{align*}
    \max_{0 < k \leq n} W_k^{(i)} \leq \max_{0 \leq k \leq n}S_k + \max_{0 < k \leq n}
    \max_{0 \leq j \leq k} \left( -S_j \right) + w_i \leq 2\max_{0
    \leq k \leq n} \left \vert S_k \right \vert + w_i,
  \end{align*}
  and this proves the result.
\end{proof}

\begin{proof}[Proof of Lemma \ref{LEM-MAX-ABS-SUM-ASYMP}]
  According to Corollary 1 of \cite{doi:10.1137/1126006}, we have that
\begin{equation}
  \Pr \left\{ \max_{0\leq n\leq m}S_{n}>x\right\} =\left( \Pr\left\{
      \max_{0\leq t\leq 1}\sigma B\left( t\right) >\frac{x}{m^{1/2}}\right\} +m
    \Pr \left\{ X > x\right\} \right) \left( 1+o\left( 1\right) \right) .
\label{Eq_LN_0}
\end{equation}
uniformly over $y\geq m^{1/2},$ as $m \rightarrow \infty$ (actually,
\cite{doi:10.1137/1126006} states that the asymptotic is valid
assuming $x/m^{1/2}\rightarrow \infty $ but the case
$x/m^{1/2}=O\left( 1\right) $ follows from the Central Limit Theorem).
Also, from the development in \cite{doi:10.1137/1126006}, because
$\Pr\{ T>x\} =o\left( \bar{B}\left( x\right) \right),$ for each $\ve
> 0,$ there is a positive integer $m_{\ve}$ such that for all $m >
m_{\ve},$ 
\begin{equation}
  \Pr \left\{ \max_{0\leq n\leq m}\left( -S_{n}\right) > x\right\} \leq
  (1+\ve)\left( \Pr \left\{ \max_{0\leq t\leq 
        1}\sigma B\left( t\right) >\frac{x}{m^{1/2}}\right\} +m\Pr \left\{
      -X>x\right\}\right) .  \label{Eq_LN_b}
\end{equation}
Additionally, since
\begin{equation*}
\Pr\left\{ \max_{0\leq n\leq m}\left\vert S_{n}\right\vert >x\right\}
\leq \Pr\left\{ \max_{0\leq n\leq m}S_{n}>x\right\} +\Pr\left\{
\max_{0\leq n\leq m}\left( -S_{n}\right) >x\right\} 
\end{equation*}
the statement of Lemma \ref{LEM-MAX-ABS-SUM-ASYMP} immediately follows
from \eqref{Eq_LN_0} and \eqref{Eq_LN_b}.
\end{proof}

\noindent \textit{Proof of Lemma \ref{Lem_Run_Max_Appers}.}
Let
\begin{align*}
  I_1(b) &:=  \mathbb{E}\left[ I\left( \max_{0\leq n\leq N_A(X) + 1} 2\left\vert S_{n}\right\vert
           >\left( \delta -\delta _{-}\right) b, N_A(X) + 1 \leq 2X\right)
           \max_{0\leq n\leq N_A(X) + 1} \left\vert S_{n}\right\vert \
           \left|\frac{}{}\right.  X > b\delta_+ \right],\\
  I_2(b) &:= \mathbb{E}\left[ I\left( \max_{0\leq n\leq N_A(X) + 1} 2\left\vert S_{n}\right\vert
           >\left( \delta -\delta _{-}\right) b, N_A(X) + 1 > 2X\right)
           \max_{0\leq n\leq N_A(X) + 1} \left\vert S_{n}\right\vert \
           \left|\frac{}{}\right.  X > b\delta_+ \right].
\end{align*}
Then our objective is to show that
$I_1(b) + I_2(b) = O(b^2\bar{B}(b)).$ This is an immediate consequence
of the following two results.

\begin{lemma}
  \label{LEM-B32-I}
  Under Assumption \ref{ASSUMP-DIST-V} with $\alpha > 2,$ and
  Assumption \ref{ASSUMP_LEFT_TAIL},
  \begin{align*}
    I_1(b) = O \left( b^2\bar{B}(b)\right).
  \end{align*}
\end{lemma}

\begin{lemma}
  \label{LEM-B32-II}
  Under Assumption \ref{ASSUMP-DIST-V} with $\alpha > 2,$ and
  Assumption \ref{ASSUMP_LEFT_TAIL},
  \begin{align*}
    I_2(b) = O \left(  \exp \left( -\nu b\right)\right), 
  \end{align*}
  for a suitable $\nu > 0.$
\end{lemma}

\begin{proof}[Proof of Lemma \ref{LEM-B32-I}]
First, observe that 
    \begin{align*}
      I_1(b) &\leq \int_{b\delta_+}^\infty \E \left[ I \left( \max_{0 \leq n \leq 2t} \left \vert S_n
               \right \vert > \frac{\delta - \delta_-}{2}b\right)
               \max_{0 \leq n \leq 2t} \left \vert S_n 
               \right \vert \right] \frac{\Pr\left\{ X \in dt 
               \right\}}{\Pr\left\{ X > b\delta_+\right\}}.
     \end{align*}
     Additionally, letting $c = (\delta - \delta_-)/2,$ observe that
     \begin{align*}
       \E \left[ I \left( \max_{0 \leq n \leq 2t} \left \vert S_n
       \right \vert > cb\right)
       \max_{0 \leq n \leq 2t} \left \vert S_n 
       \right \vert \right] &= cb \Pr \left\{ \max_{0 \leq n \leq 2t}
                              \left \vert S_n \right \vert >
                              cb\right\} + \int_{cb}^\infty \Pr\left\{
                              \max_{0 \leq n \leq 2t} \left \vert S_n
                              \right \vert > u \right\} du 
     \end{align*}
     Therefore, due to \eqref{tail_equiv_F_B},
     \begin{align}
       I_1(b)  = O \left( \frac{1}{\bar{B}(b\delta_+)}  
       \int_{b\delta_+}^\infty \left(b \Pr \left\{ \max_{0 \leq n \leq 2t} 
       \left \vert S_n \right \vert > cb\right\} +  
       \int_{cb}^\infty \Pr\left\{ \max_{0 \leq n \leq 2t} \left \vert
       S_n  \right \vert > u \right\} du \right) \Pr \{ X \in  dt\}
       \right). 
      \label{INTER-GG2-I1}
     \end{align}
     Due to the applicability of the uniform asymptotic presented in
     Lemma \ref{LEM-MAX-ABS-SUM-ASYMP} in the region $2t \leq c^2b^2,$
     and because of the applicability of Central Limit Theorem in the
     region $2t > c^2b^2,$ we obtain
     \begin{align*}
       &\int_{b\delta_+}^\infty\Pr \left\{ \max_{0 \leq n \leq 2t} 
         \left \vert S_n \right \vert > cb\right\} \Pr\{ X \in dt \}
         \nonumber\\
       &\quad\quad= 
         O\left(\int_{b \delta_+}^\infty \Pr \left\{ \max_{0 \leq s \leq
         1}\sigma  \left\vert B(s) 
         \right \vert > \frac{cb}{\sqrt{2t}}\right\} \Pr\{ X
         \in dt\} + \int_{b\delta_+}^{\frac{c^2b^2}{2}} t \Pr\left\{
         \left \vert X \right \vert > cb\right\} \Pr\{X \in
         dt\}\right) \nonumber\\
       &\quad\quad=
         O\left( b \int_{b\delta_+}^\infty \frac{\Pr\left\{ X >
         t\right\}}{\sqrt{t^3}} \exp\left( -\frac{c^2 
         b^2}{4\sigma^2t}\right) dt + \Pr \{ \vert X \vert > cb \} \E
         \left[ X I \left( X \in \left[ b\delta_+,
         \frac{c^2b^2}{2}\right] \right)\right]\right) 
     \end{align*}
     due to integration by parts. Now, one can apply Lemma
     \ref{LEM-BSQ-SIMP-TERM} to evaluate the first integration, and
     Karamata's theorem for the second integration, together with the
     observation that $\Pr\{ \vert X \vert > x\} = O(\bar{B}(x)),$ to
     obtain 
      \begin{align}
        \int_{b\delta_+}^\infty\Pr \left\{ \max_{0 \leq n \leq 2t} 
        \left \vert S_n \right \vert > cb\right\} \Pr\{ X \in dt \}
        = O\left( \bar{B}\left( b^2 \right) +
        \bar{B}(b)\times b\bar{B}(b)\right)
         \label{INTER-GG2-I1-II}
      \end{align}
      On similar lines of reasoning using Lemma
      \ref{LEM-MAX-ABS-SUM-ASYMP}, again via careful integration by
      parts and subsequent application of Lemma
      \ref{LEM-BSQ-SIMP-TERM} and Karamata's theorem, one can derive
     \begin{align*}
       &\int_{b\delta_+}^\infty \int_{cb}^\infty \Pr\left\{ \max_{0
         \leq n \leq 2t} \left \vert 
         S_n  \right \vert > u \right\} du \Pr\{ X \in dt\}\\
       &\quad\quad=
         O\left( \int_{b \delta_+}^\infty \int_{cb}^\infty \Pr \left\{ \max_{0 \leq s \leq
         1}\sigma  \left\vert B(s) 
         \right \vert > \frac{u}{\sqrt{2t}}\right\} du \ \Pr\{ X
         \in dt\} + \int_{b\delta_+}^\infty\int_{\sqrt{2t}}^\infty t
         \Pr\left\{ \left \vert X \right \vert > u\right\} du \ \Pr\{X \in
         dt\} \right)\\
       &\quad\quad=
         O\left( \int_{b\delta_+}^\infty \frac{\Pr\{X > t\}}{\sqrt{t}}
         \exp \left( -\frac{c^2b^2}{4\sigma^2 t}\right)dt +
         \int_{\sqrt{2b\delta_+}}^{\infty}\int_{b\delta_+}^{\frac{u^2}{2}}
         t\Pr\{ X \in dt\} \Pr \{ \vert X \vert > u\}\right) \\
         &\quad\quad = O\left( b\bar{B} \left( b^2 \right) +
           b\bar{B}(b) \times b\bar{B}(b)\right). 
     \end{align*}
     This bound, along with \eqref{INTER-GG2-I1},
     \eqref{INTER-GG2-I1-II} and the observation that
     $\bar{B}(b\delta_+) =
     \Theta(\bar{B}(b)),$ prove Lemma \ref{LEM-B32-I}.
\end{proof}

\begin{proof}[Proof of Lemma \ref{LEM-B32-II}] 
  Since $T_1 + \ldots + T_{N_A(t)} \leq t$ (follows from the
  definition of $N_A(t)$) and $V_0 := 0,$
  \begin{align*}
    \E &\left[ I\left( N_A(t) + 1 >
         2t\right) \max_{0 \leq n \leq N_A(t) + 1} \left\vert S_n
         \right\vert\right] \leq \E \left[ I\left( N_A(t)  > 
         2t -1 \right) \sum_{n=1}^{N_A(t) + 1} \left(  V_{n-1} + T_{n}
         \right)  \right]\\
       &\quad\quad\quad\quad \quad\quad\quad\quad \leq \E \left[
         I\left( N_A(t)  >
         2t - 1 \right) \left(\sum_{n=1}^{N_A(t)}  V_{n} +
         t + T_{N_A(t) + 1} \right)\right]\\
       &\quad\quad\quad\quad \quad\quad\quad\quad \leq \E V \times \E
         \left[ N_A(t) I 
         \left( N_A(t) > 
         2t - 1\right)\right] + \left( t + \E T \right) \Pr \left\{ N_A(t) >
         2t - 1\right\}\\
       &\quad\quad\quad\quad \quad\quad\quad\quad \leq C_1 t \Pr
         \left\{ N_A(t) >  2t-1 \right\} + C_2 \int_{2t-1}^\infty
         \Pr\left\{ N_A(t) > s\right\} ds
  \end{align*}
  for suitable positive constants $C_1$ and $C_2$ independent of $t.$
  Here, note that the penultimate inequality is simply due to the
  independence between $V_n$ and $T_n$ for $n \geq 1.$ Therefore,
  \begin{align}
    I_2(b) &\leq \int_{b\delta_+}^\infty  \E \left[ I\left( N_A(t) + 1 >
             2t\right) \max_{0 \leq n \leq N_A(t) + 1} \left\vert S_n
             \right\vert\right] \frac{\Pr \left\{ X \in dt
             \right\}}{\Pr\left\{ X > b \delta_+ \right\}} \nonumber\\
           &\leq C_1\int_{b\delta_+}^\infty t \Pr\left\{ N_A(t) > 2t-1 
             \right\} \frac{\Pr \left\{ X \in dt
             \right\}}{\Pr\left\{ X > b \delta_+ \right\}}  + C_2
             \int_{b\delta_+}^\infty\int_{2t-1}^\infty \Pr \left\{
             N_A(t) > s \right\} ds \frac{\Pr \left\{ X \in dt 
             \right\}}{\Pr\left\{ X > b \delta_+ \right\}} \nonumber\\  
           &\leq \Pr \left\{ N_A \left(b\delta_+ \right) > 2b \delta_+-1
             \right\}  \left( C_1\int_{b\delta_+}^\infty t \frac{\Pr \left\{ X \in dt
             \right\}}{\Pr\left\{ X > b \delta_+ \right\}} 
             + C_2 \int_{2b\delta_+-1}^\infty \frac{\Pr \left\{ X > \frac{s}{2}
             \right\}}{\Pr\left\{ X > b \delta_+ \right\}} ds
             \right) \label{I2-INTER}, 
  \end{align}
  where we have used a simple change of order of integration to arrive
  at the above conclusion. Since $N_A(x)/x \rightarrow 1$ as
  $x \rightarrow \infty,$ the event $\{ N_A(b\delta_+) > 2b\delta_+-1\}$
  is a large deviations event with probability exponentially decaying
  in $b,$ whereas the sum appearing in the parenthesis in
  \eqref{I2-INTER} is $O(b)$ due to Karamata's theorem. This proves
  the claim that $I_2(b) = O(\exp(-\nu b))$ for a suitable constant
  $\nu > 0.$
\end{proof}
\noindent As mentioned earlier, Lemmas \ref{LEM-B32-I} and
\ref{LEM-B32-II}, together complete the proof of Lemma
\ref{Lem_Run_Max_Appers}. 
\hfill{$\Box$} 

\begin{proof}[Proof of Lemma \ref{LEM-BSQ-SIMP-TERM}]
  Due to Potter's bounds \eqref{POTT_BND}, given $\ve > 0,$ we have
  \begin{align*}
    \int_{b}^\infty \frac{v(t)}{v \left( b^2 \right)} \exp\left(-\frac{cb^2}{t}
    \right) dt = O \left(\int_b^\infty \left(\frac{b^2}{t}\right)^{\alpha -
    \ve}  \exp\left(-\frac{cb^2}{t} \right) dt \right)
  \end{align*}
  for all suitably large values of $b.$ Changing variables
  $u = b^2/t,$ we obtain
  \begin{align*}
    \int_{b}^\infty \frac{v(t)}{v \left( b^2 \right)} \exp\left(-\frac{cb^2}{t}
    \right) dt = O \left( b^2 \int_0^\infty u^{\alpha-2-\epsilon}
    \exp(-cu) du \right) = O\left( b^2 \right) 
  \end{align*}
  for all $\epsilon$ small enough such that $\alpha-2 - \epsilon > 0$,
  and this verifies the claim.
\end{proof}

\begin{proof}[Proof of Lemma \ref{LEM-INF-VAR-TERM2}]
Letting $\bar{c} = (\delta-\delta_-)/2,$ observe that
\begin{align}
  &\E \left[ I \left(\max_{0 \leq n \leq 2X} \left \vert S_n \right \vert >
    \bar{c} b \right) \max_{0 \leq n \leq 2X} \left \vert S_n \right
    \vert \ \left| \frac{}{} \right. X > b\delta_+ \right] \nonumber\\
  &\quad\quad= \bar{c}b \Pr \left\{ \max_{0 \leq n \leq 2X} \left
    \vert S_n \right \vert   > \bar{c} b \ \left \vert \frac{}{}
    \right. X > b\delta_+ \right\} + \int_{\bar{c}b}^\infty  \Pr \left\{ \max_{0
    \leq n \leq 2X} \left \vert S_n \right \vert   > u \ \left \vert \frac{}{}
    \right. X > b\delta_+ \right\} du \nonumber\\
  &\quad\quad \leq  3\bar{c}b \Pr\left\{ Z_* >
    \frac{\bar{c}b}{\left(2cX\right)^{\frac{1}{\alpha}}} \ \left \vert \frac{}{}
    \right. X > b\delta_+\right\} +
    3\int_{\bar{c}b}^\infty  \Pr\left\{ Z_* >
    \frac{u}{\left(2cX\right)^{\frac{1}{\alpha}}} \ \left \vert \frac{}{}
    \right. X > b\delta_+ \right\} du
    \label{INF-VAR-LAST-INTER}
\end{align}
because of Lemma \ref{LEM-MAX-SUM-ASYMP-INF-VAR}. Since
$\Pr\{ X > x\} \sim cx^{-\alpha}$ as $x \rightarrow \infty,$ after
simple integration using Karamata's theorem \eqref{KARAMATA-II}, one 
can show that
\begin{align*}
  \bar{c}b \Pr\left\{ Z_* >
  \frac{\bar{c}b}{\left(2cX\right)^{\frac{1}{\alpha}}} \ \left \vert
  \frac{}{} \right. X > b\delta_+\right\} &= \bar{c}b \E \left[
                                            \frac{\Pr\left\{ X > \frac{1}{2c} \left( \frac{\bar{c} b}{Z_*}
                                            \right)^{\alpha}\right\}}{\Pr
                                            \left\{X >
                                            b\delta_+\right\}} \wedge 
                                            1\right] = O\left( b^2\bar{B} (b) \right), \text{ and }\\
  \int_{\bar{c}b}^\infty  \Pr\left\{ Z_* >
  \frac{u}{\left(2cX\right)^{\frac{1}{\alpha}}} \ \left \vert \frac{}{}
  \right. X > b\delta_+ \right\} du &= \int_{\bar{c} b}^\infty \left(
                                      \frac{\Pr\left\{ X >
                                      \frac{u^\alpha}{2c Z_*^\alpha} 
                                      \right\}}{\Pr\left\{ X >
                                      b\delta_+\right\}} \wedge
                                      1\right) du = O \left(
                                      b^2\bar{B}(b) \right). 
\end{align*}
 Therefore,  due to \eqref{INF-VAR-LAST-INTER}, we obtain 
\begin{align*}
  \E \left[ I \left(\max_{0 \leq n \leq 2X} \left \vert S_n \right \vert >
  \bar{c} b , \ N_A(X) + 1 \leq 2X \right) \max_{0 \leq n \leq 2X} \left \vert S_n \right
  \vert \ \left| \frac{}{} \right. X > b\delta_+ \right] = O\left( b^2
  \bar{B}(b)\right).
\end{align*}
On the other hand, the component corresponding to the large deviations
event $\{ N_A(X) + 1 \geq 2X \}$ is handled similar to Lemma
\ref{LEM-B32-II}, and this upper bounding procedure results in 
\begin{align*}
  \E \left[ I \left(\max_{0 \leq n \leq 2X} \left \vert S_n \right \vert >
  \bar{c} b , \ N_A(X) + 1 > 2X \right) \max_{0 \leq n \leq 2X} \left \vert S_n \right
  \vert \ \left| \frac{}{} \right. X > b\delta_+ \right] = O\left(
  \exp(-\nu b) \right),
\end{align*}
for some $\nu > 0.$ The last two upper bounds are enough to conclude
the statement of Lemma \ref{LEM-INF-VAR-TERM2}. 
\end{proof}

\bibliographystyle{acm}
\bibliography{asymp}

\end{document}